\documentclass{amsart} 
\usepackage{amsmath,amssymb,mathrsfs, xcolor}
\usepackage{graphicx} 
\usepackage[title]{appendix}
\usepackage{comment}
\usepackage[normalem]{ulem}
 \usepackage{url}

\newcommand{\RN}[1]{%
  \textup{\uppercase\expandafter{\romannumeral#1}}%
}
\newcommand{\fillellipse}{\raisebox{0.2ex}{\scalebox{2.2}[1.4]{\textcolor{blue}{$\bullet$}}}}
\DeclareMathOperator{\Tr}{Tr}

\newtheorem{cor}{Corollary}[section]
\newtheorem{lem}{Lemma}[section]
\newtheorem{prop}{Proposition}[section]

\newtheorem{rem}{Remark}[section]

\newtheorem{problem}{Problem}[section]

\newtheorem{thm}[problem]{Theorem}

\title{Capillary minimal slicing and scalar curvature rigidity}

\author{Dongyeong Ko and Xuan Yao}
\address{Department of Mathematics, Massachusetts Institute of Technology, Cambridge, MA 02139}
\email{dyko@mit.edu}
\address{Department of Mathematics, University of Chicago, Chicago, IL, 60637}
\email{dorisxyao@uchicago.edu}
\begin{document}

\begin{abstract}
    We develop minimal slicing via capillary hypersurfaces to understand scalar curvature metric on manifolds bounded below with boundary. In particular, in dimension $4$, we prove following comparison and rigidity statement: given a compact Riemannian $4$-manifold $(M^4,g)$ with a mean convex boundary whose boundary is diffeomorphic to boundary of a connected convex domain in $\mathbb R^4$, if the scalar curvature is non-negative and the scaled mean curvature comparison holds along the boundary, then $M$ is isometric to the Euclidean domain. Moreover, we prove a comparison and rigidity result of manifolds with Listing type scaling invariant scalar and mean curvature comparison : $R g \ge -12 g_{\mathbb{H}^{4}}$ and $\hat{H}^{2} g \ge \hat{H}_{0} g_{\mathbb{H}^{4}}$, where $\hat{H}$ is a shifted mean curvature depending on the scale of manifold. This proves the traced version of Gromov's conjecture on comparison and rigidity of hyperbolic convex domains.
\end{abstract}

\maketitle

\section{Introduction}
Studying geometry and topology of manifolds with positive scalar curvature has been a fundamental topic in differential geometry, which appears both in geometry and general relativity naturally. Among the questions involving with scalar curvature including comparison theorems, positive mass theorem, topological classification, interactions between interior scalar curvature and boundary mean curvature have given clues on the local and global geometry of positive scalar curvature metric.

There have been several major approaches to understand manifolds with positive scalar curvature such as Dirac operator method, the variational method of area functional and harmonic functions. In particular, the variational approach with minimal hypersurfaces has played an important role in developing the theory, while many fundamental questions on manifolds with scalar curvature lower bound remains unknown. More specifically, many known methods to deduce the comparison and rigidity statement via variational methods rely on the stability property of a minimizer of area functional of hypersurfaces. The stability inequality gives comparison and rigidity of geometric and topological quantities. In particular, obtaining rigidity statements through variational methods have been successful in three dimensional geometry, where we have more direct ways to extract geometric and topological quantities using classical theorems such as Gauss-Bonnet. However, despite recent progress on the control of quantities in higher dimensions, for instance, $\sigma$-constant introduced by Schoen \cite{schoen2006variational} and Kobayashi \cite{kobayashi1986scalar} for Yamabe problem and the recent work of Stern \cite{stern2022scalar} by the level set of harmonic maps, only few methods can be employed to understand scalar curvature from minimizers. Moreover, almost all of these methods can be applied only in the manifold without boundary settings.

The Dirac operator method has been successful to prove major results when the manifold is spin, which is initiated by Lichnerowicz \cite{lichnerowicz1963geometrie}, Atiyah and Singer \cite{atiyah1963index} in 1960s. The theory was further developed by Hitchin \cite{hitchin1974harmonic}, and Gromov and Lawson \cite{gromov1980classification, gromov1983positive, gromov1980spin}. In the boundary setting, for the rigidity of Riemannian polyhedra, namely dihedral rigidity conjecture, Gromov proved the conjecture for cubical polyhedra by doubling and reducing to Geroch conjecture in high dimensions, and sketched a proof, which can be found in \cite{gromov2014dirac}, \cite{gromov2021lecturesscalarcurvature}, \cite{gromov2023convexpolytopesdihedralangles}. More recently, Wang-Xie-Yu \cite{wang2021gromov} developed spin geometry on cornered manifold to prove the dihedral rigidity conjecture; see also \cite{wang2026gromovsdihedralrigidityconjecture}. Brendle proved the dihedral rigidity conjecture with the matching angle hypothesis combining harmonic spinor and smooth-out method \cite{brendle2023matchingangle} (See also Brendle-Wang \cite{brendle2023gromovsrigiditytheorempolytopes} and Bi \cite{bi2026dihedralrigidityconvexpolytopes}).

On the other hand, minimal hypersurface theory has led progress on major questions on scalar curvature and positive mass theorem. Remarkably, on a conjecture of Geroch, which conjectures that the torus $T^{n}$ does not admit a metric of positive scalar curvature, Schoen-Yau proved the conjecture in dimension $3$ first in \cite{schoen1979existence} and extended to dimension $4$ to $7$ by developing the descent on homology classes in \cite{schoen1987structure} (Gromov and Lawson proved the conjecture with the spinor method on spin manifolds in all dimensions in \cite{gromov1983positive}). Following an outline of proof in dimension $4$ by Schoen-Yau \cite{schoen1987structure}, the non-existence of metrics of positive scalar curvature on closed aspherical manifolds, also known as the $K(\pi,1)$ conjecture or aspherical conjecture, was proven in dimension $4$ and $5$ by Chodosh and Li in \cite{chodosh2024generalized} and by M. Gromov in \cite{gromov2020no} by generalized soap bubble method which is a generalized version of minimal hypersurfaces. The generalized soap bubble technique, also known as the $\mu$-bubble and prescribed mean curvature functional, was first introduced to the study of the positive scalar curvature problem by Gromov \cite{gromov2021lecturesscalarcurvature}, the existence and regularity (in low dimensions) of the $\mu$-bubble minimizer was first rigorously proven by Zhu \cite{zhu2021width}. See also \cite{zhou2019cmc}, \cite{zhou2020pmc} for a min-max construction of the prescribed mean curvature functional.

The weighted minimal slicing method developed by Schoen and Yau in \cite{schoen2017positive} has enabled to reduce problems on higher dimensions to problems in dimension $2$, in particular in closed manifold settings, and this extended the positive mass theorem to all dimensions by analyzing singular minimal hypersurfaces. Recently, Brendle-Hirsch-Johne extended weighted minimal slicing techniques to prove the generalized Geroch conjecture using intermediate curvature \cite{brendle2024generalization} (See also \cite{chu2022rigidity} for the corresponding rigidity statement). Mazurowski-Wang-Yao \cite{mazurowski2026topology} combined the weighted minimal slicing and the generalized soap bubble method in the study of topological obstructions of positive intermediate curvature.

Analyzing capillary minimal surfaces to understand interaction between mean curvature and scalar curvature is natural, in the sense that capillary minimal surfaces prescribe the intersection angle between minimal surfaces and boundary of an ambient manifold. Remarkably, Li settled the dihedral rigidity conjecture for conical and prism-type polyhedra by applying capillary minimal surfaces in dimension $3$ \cite{li2020polyhedron} (See \cite{li2024dihedral} for the generalization to $n$-prisms).

For the rigidity results of smooth domains, Gromov proved the extremality for the standard unit ball with capillary minimal surfaces in his lecture note \cite{gromov2021lecturesscalarcurvature}. Chai-Wang \cite{CW1} generalized the rigidity theorem to $\mathbb{S}^{1}$-rotationally symmetric domains (See \cite{chai2024dihedral} in hyperbolic setting and \cite{chai2025scalar} in a $3$-dimensional warped product). Recently, Wu applied capillary surfaces with prescribed contact angle to obtain fill-in result and a bandwidth estimate in \cite{wu2025capillary} (See \cite{wu2025rigidity} also). Wang-Wang-Zhu proved scalar-mean curvature rigidity theorem for unit balls by applying capillary soap bubbles in \cite{wang2024scalar}.  The authors provided variational proof of comparison and rigidity of weakly convex domains with boundary diffeomorphism and the scaling mean curvature comparison introduced by Gromov, which we will discuss further, in dimension $3$ in \cite{koyao}, where Bär-Brendle-Chow-Hanke \cite{bar2023rigidity} proved the rigidity theorems for initial data sets via spinor approach which covers weakly convex domain cases.

Proving full regularity of a minimizer of capillary functional provides the comparison and rigidity statement from the slicing arguments. The classical regularity result of Taylor in \cite{taylor1977boundary} gives a smoothness of minimizing surfaces of the capillary functional in dimension $3$ (See De Philippis-Maggi \cite{philippis2015regularity} for more general settings). Recently, Chodosh-Edelen-Li extended the regularity result to dimension $4$ in \cite{chodosh2024improved} (See \cite{chodosh2025weiss} also). In contrast, for free boundary cases, it is proven that the singular set of a minimizer of the area functional has codimension at least $7$ by Grüter \cite{gruter1987optimal}. More recently, Firester-Tsiamis-Wang extended the regularity theory by providing non-flat area-minimizing capillary cones in dimension $8$ or higher in \cite{firester2026area}.

There are recent construction result of critical point of perturbed area functional such as prescribed mean curvature surface and capillary surface via min-max approach (See Zhou-Zhu \cite{zhou2018existence}, Li-Zhou-Zhu \cite{li2025min}, and Ko \cite{ko2023min}).

In the non-zero scalar curvature setting, there is a celebrated result of Llarull \cite{llarull1998sharp} that if there is an area non-increasing map $f : M \rightarrow \mathbb{S}^{n}_{round}$ for a closed and connected spin manifold $M$, then $M$ is isometric to the round sphere $\mathbb{S}^{n}_{round}$. Later Listing \cite{listing2010scalar} generalized the trace version of this rigidity problem (See the trace $\Lambda^{2}df$-Inequality in Gromov's lecture note \cite{gromov2021lecturesscalarcurvature}). While Cecchini-Wang-Xie-Zhu recently fully proved Llarull's theorem in dimension $4$, Llarull and Listing type theorem for non-spin manifolds are widely open in higher dimension. Brendle-Marques-Neves \cite{brendle2011deformations} disproved Min-Oo's conjecture on hemisphere rigidity by constructing a counterexample by perturbing the standard metric on the hemisphere $\mathbb{S}^{n}_{+}$ in such a way that the scalar curvature increases at each point and the mean curvature of the boundary becomes positive.

There have been progress on comparison and rigidity problems in the nonpositive curvature setting. In \cite[P. 109, Section 3.5]{gromov2021lecturesscalarcurvature}, Gromov conjectured the comparison and rigidity of hyperbolic convex domains (in the sense of Euclidean domain). Chai-Wang \cite{chai2024dihedral} formulated the dihedral rigidity of $3$-dimensional hyperbolic domains proved it for certain type of domains. Chai-Wan \cite{chai2024scalar} extended the previous work to the warped product setting and proved the conjecture of Gromov with conditions $R \ge -n(n-1)$, $H_{g} \ge H_{b}$, and $g|_{\partial \Omega} \ge b|_{\partial \Omega}$.

One natural setting to understand the scalar and mean curvature comparison and rigidity is on smooth convex domains. In this regards, Gromov posed and proved the following $\fillellipse$-inequality in his lecture note \cite{gromov2021lecturesscalarcurvature}:
\begin{thm} [\cite{gromov2021lecturesscalarcurvature}, Section 3.1.\RN{2}] \label{blueellipseineq}
    Let $V$ be a Riemannian manifold diffeomorphic to the $n$-ball the boundary of $V$ of which has positive mean curvature $H_{\partial V} >0$, let $\underline{V} \subset \mathbb{R}^{n}$ be a convex domain with smooth boundary and let $f: \partial V \rightarrow \partial \underline{V}$ is a diffeomorphism. If $R_{V} \ge 0$, then the differential of $f$ cannot be everywhere strictly smaller than the ratio of the mean curvatures of the two boundaries: there exists a point $v \in \partial V$, such that
    \begin{align} \label{mccomparisoncondition}
    ||df(v)|| \ge \frac{H_{\partial V}(f(v))}{H_{\partial \underline{V}}(v)}.
    \end{align}
\end{thm}
The scaling invariant mean curvature condition (\ref{mccomparisoncondition}) is introduced in Gromov's lecture note with the terminology of `normalization of metrics by mean curvatures'. The statement suggests the sharp comparison between interior scalar curvature and boundary mean curvature of convex domains, and the proof relies on the spin structure and Dirac operator.

\subsection{Main results}
We now describe our capillary minimal slicing which provides applications in comparison and rigidity statements, in particular, in non-spin settings. We adopt Schoen-Yau's weighted minimal slicing technique to manifold with boundary. In the slicing $X \subset X \times S^{1} \subset ... \subset X \times T^{n-k}$ of a closed manifold $X^{k} \times T^{n-k}$, the existence of nontrivial minimizer is guaranteed by the nontrivial topology of the torus part $T^{n-k}$. However, in our manifold with boundary setting, we lack the topological information of manifold since we only know that $\partial M$ is diffeomorphic to an $(n-1)$-dimensional sphere. For this reason, we impose the contact angle functions for each minimal slicing which replace topological constraints in the minimal slicing of closed manifolds. We deduce the existence of nontrivial minimizers of capillary area functional from the mean curvature comparison and the angle comparison. 

Let us first define the slicing angle functions. Suppose $\bar{M}^{n+2}$ to be a Euclidean convex domain with smooth boundary and we take orthogonal coordinates $(x_{1}, ..., x_{n+2})$.  Denote $\bar{\Sigma}_{t_{1}, ..., t_{i}}^{n+2-i}$ to be a codimension $i$ slice $\bar{M} \cap( \cap_{j=1}^{i} \{ x_{n+3-j} = t_{j} \})$ and assume $x \in \bar{\Sigma}_{t_{1}, ..., t_{i}}$. We define the angle functions $\{\bar{\theta}_{i}\}_{i=1, ..., n}$ at $x$ by the dihedral angle $\bar{\theta_{i}}$ between $\bar{\Sigma}_{t_{1}}$ and $\partial \bar{M}$ for $i=1$ and the dihedral angle between $\bar{\Sigma}_{t_{1}, ..., t_{i}}$ and $\partial \bar{\Sigma}_{t_{1}, ..., t_{i-1}}$ otherwise, namely the dihedral angle between codimension $i$ level sets and the boundary of codimension $i-1$ level sets in a Euclidean domain. We call the angle functions $\{\bar{\theta}_{i}\}_{i=1,2 ..., n}$ are associated angle functions with the coordinate $(x_{1}, ..., x_{n+2})$. Given angle functions $\{\bar{\theta}_{i}\}_{i=1, ..., n}$ on a Euclidean domain $\partial \bar{M}$, the angle function on $\partial M$ is given by the pullback of the angle functions by boundary diffeomorphism i.e. $\{ f^{\sharp} \bar{\theta}_{i}\}_{i=1, ..., n}$ are angle functions on $\partial M$ where $f : \partial M \rightarrow \partial \bar{M}$ is a boundary diffeomorphism.

In dimension $4$, our main slicing statement is as follows. 

\begin{thm}\label{main slicing thm}
    Let $M^{4}$ be a manifold with boundary of dimension $4$ with non-negative scalar curvature $R \ge 0$, whose boundary is diffeomorphic to the boundary of a Euclidean convex domain $\bar{M}^{4}$. Moreover, suppose that $\partial M$ satisfies the rescaled mean curvature comparison $H_g^2g|_{\partial M}\geq H_0^2g_{E}|_{\partial \bar{M}}$. Then there exist orthogonal coordinates $(x_{1}, ..., x_{4})$ such that there exists a stable weighted capillary minimal slicing $\Sigma_{2}^{2} \subset \Sigma_{1}^{3} \subset M^{4}$ with the associated angle functions $\{\bar{\theta}_{i}\}_{i=1,2}$ satisfying the following conditions.
    \begin{enumerate}
        \item For each $k=1,2$, $\Sigma_{k}$ is an embedded capillary minimal hypersurface in $\Sigma_{k-1}$ with a contact angle $\bar{\theta}_{k}$. Moreover, $\Sigma_{k}$ is a stable critical point of the $\rho_{k-1}$-weighted capillary area functional
        \[
            \mathcal A_{\bar{\theta}_k}(U_k):=\int_{\partial_{rel}U_k}\rho_{k-1}d \mathcal{H}^{4-k}-\int_{\partial U_n\cap \partial \Sigma_{k-1}}\rho_{k-1} \cos\bar{\theta}_k d \mathcal{H}^{4-k},
        \]
        where $\rho_{0}=1$ and $\Sigma_{0} = M$.
        \item For each $k=1,2$, the function $f_k|_{\Sigma_{k}} \in C^{\infty}(\Sigma_{k})$ is a first eigenfunction of the stability operator associated with the $\rho_{k-1}=\Pi_{i=1}^{k-1}f_i$ weighted capillary area functional.
    \end{enumerate}
\end{thm}
\begin{rem} 
\begin{enumerate}
    \item The proof of Theorem \ref{main slicing thm} relies on the smoothness of minimizers of the capillary area functional and the curvature estimates which come from the Bernstein problem of capillary area minimizing cone. The best known regularity is the full regularity in ambient dimension $4$, namely by the work of Chodosh-Edelen-Li \cite{chodosh2024improved}. 
    \item In this paper, we prove the existence of weighted capillary minimal slicing in dimension $4$. For the existence in higher dimensions, we may need to impose further assumptions on the Euclidean domain $\bar{M}$.
    \end{enumerate}
\end{rem}
We extend the main slicing statement to the manifolds with negative scalar curvature lower bound with scaled mean curvature comparison with boundary. Our condition for the slicing theorem is on the scaling invariant scalar curvature comparison, to the author's knowledge, this is the first result on interior scaling invariant scalar curvature comparison with variational method. 

More specifically, we take a scaling invariant scalar curvature comparison $R g \ge R_{\mathbb{H}^{n+2}} g_{hyp}$, where $R_{\mathbb{H}^{n}} = -n(n-1)$. For each point $x \in \Omega$, the minimal stretch $\rho : \Omega \rightarrow \mathbb{R}_{+}$ is defined by
\[
\rho^{2}(x) = \min _{|v |_{g} =1} g_{Hyp}(v,v).
\]
Then we take a scaling factor by $c = \max_{x \in \Omega} \rho(x)$. Now we state our slicing result in the negative curvature setting.
\begin{thm}\label{main slicing thm :hyp}
    Let $M^{4}$ be a manifold with boundary of dimension $4$ with non-negative scalar curvature $R g \ge -12 g_{0}$, which is diffeomorphic to a hyperbolic convex ball $\bar{M}^{4}$. Moreover, suppose that $\partial M$ satisfies the rescaled mean curvature comparison $(H + 3c \cos \bar{\theta}_{1})^{2} g \ge (H_{\mathbb{H}^{4}} + 3 \cos  \bar{\theta}_{1})^{2} g_{\mathbb{H}^{4}}$ and $H + 3c \cos \bar{\theta}_{1} \ge 0$. Then there exist orthogonal coordinates $(x_{1}, ..., x_{4})$ such that there exists a stable weighted capillary minimal slicing $\Sigma_{2}^{2} \subset \Sigma_{1}^{3} \subset M^{4}$ with the associated angle functions $\{\bar{\theta}_{i}\}_{i=1,2}$ satisfying the following conditions.
    \begin{enumerate}
        \item $\Sigma_{1}$ is a stable embedded capillary CMC hypersurface in $M$ with a contact angle $\bar{\theta}_{1}$.
        \item $\Sigma_{2}$ is an embedded capillary minimal hypersurface in $\Sigma_{1}$ with a contact angle $\bar{\theta}_{2}$. Moreover, $\Sigma_{2}$ is a stable critical point of the $f_{1}$-weighted capillary area functional
        \[
            \mathcal A_{\bar{\theta}_2}(U_2):=\int_{\partial_{rel}U_2}\rho_{1}d \mathcal{H}^{2}-\int_{\partial U_2\cap \partial \Sigma_{1}}\rho_{1} \cos\bar{\theta}_2 d \mathcal{H}^{2}.
        \]
 The function $f_1|_{\Sigma_{1}} \in C^{\infty}(\Sigma_{1})$ is a first eigenfunction of the stability operator associated with the perturbed capillary area functional:
 \[
 \mathcal A_{\bar{\theta}_1}(\Omega):=|\partial^i\Omega|-c(n+1)\text{Vol}_g(\Omega)-\int_{\partial^b \Omega}\cos\bar{\theta}_1 d\mathcal H^{n+1}.
 \]
    \end{enumerate}
\end{thm}

By analyzing the stable minimal capillary slices constructed in Theorem \ref{main slicing thm}, we obtain the following rigidity theorem. Since the estimates we proven in Section 2 and 3 works dimension-wise, we state the theorem in general dimensions. In the following statement, we assume the existence and regularity of a minimizer of the capillary area functional, which follows the curvature estimate. 

\begin{thm} \label{main rigidity in general dimension}
    Let $M^{n+2}$ be a manifold with boundary of dimension $n+2$ with non-negative scalar curvature $R \ge 0$, whose boundary is diffeomorphic to the boundary of a Euclidean convex domain $\bar{M}^{n+2}$. Moreover, suppose that $\partial M$ satisfies the rescaled mean curvature comparison $H_g^2g|_{\partial M}\geq H_0^2g_{E}|_{\partial \bar{M}}$. Assume the regularity of minimizer of capillary area functional and that there exists a stable weighted capillary minimal slicing  $\Sigma_{n}^{2} \subset \Sigma_{n-1}^{3} \subset ... \subset M^{n+2}$ with the associated angle functions $\{\bar{\theta}_{i}\}_{i=1,...n}$. Then $(M,g)$ is isometric to $(\bar{M}, g_{E})$.
\end{thm}
We obtain the corresponding theorem in the hyperbolic setting:
\begin{thm} \label{main rigidity in general dimension : hyp}
    Let $M^{n+2}$ be a manifold with boundary of dimension $n+2$ with scalar curvature lower bound $R g \ge -(n+2)(n+1)g_{\mathbb{H}^{n+2}}$, which is diffeomorphic to a hyperbolic convex ball $\bar{M}^{n+2} \subset \mathbb{H}^{n+2}$. Moreover, suppose that $\partial M$ satisfies the shifted mean curvature comparison $(H + c(n+1) \cos \bar{\theta}_{1})^{2} g \ge (H_{\mathbb{H}^{n+2}} + (n+1) \cos  \bar{\theta}_{1})^{2} g_{\mathbb{H}^{n+2}}$ and $H + c(n+1) \cos \bar{\theta}_{1} \ge 0$, where $c$ is the scaling factor we introduced earlier. Assume the regularity of minimizer of capillary area functional and that there exists a stable weighted capillary minimal slicing  $\Sigma_{n}^{2} \subset \Sigma_{n-1}^{3} \subset ... \subset M^{n+2}$ with the associated angle functions $\{\bar{\theta}_{i}\}_{i=1,...n}$. Then $(M,g)$ is isometric to $(\bar{M}, g_{E})$.
\end{thm}

\begin{rem}
    For the boundary mean curvature comparison, we use a shifted mean curvature $H+c(n+1) \cos \bar{\theta}_{1}$ instead of the original mean curvature. This is because a horosphere foliation in the hyperbolic metric induces a natural shifted of a mean curvature on boundary. We define this to be a horospherically shifted mean curvature in Section \ref{sec:hyp setting}. Note that a domain with $H_{\mathbb{H}^{n+2}} + (n+1) \cos  \bar{\theta}_{1}$ in the standard hyperbolic metric is a convex domain in a Euclidean metric. 
\end{rem}
By combining Theorem \ref{main slicing thm} and \ref{main rigidity in general dimension} and applying this to dimension $4$ cases, we obtain our main rigidity theorem in dimension $4$ for nonnegative scalar curvature.
\begin{thm}\label{main rigidity thm}
    Suppose $(M^4,g)$ is a compact Riemannian manifold with non-negative scalar curvature $R \ge 0$ and assume that $(\partial M,g|_{\partial M})$ is mean convex and diffeomorphic to the boundary $\partial \bar{M}$ of a smooth convex domain $(\bar{M},g_E)$ in $\mathbb R^4$. Furthermore, we assume that the rescaled mean curvature comparison $H_g^2g|_{\partial M}\geq H_0^2g_{E}|_{\partial \bar{M}}$ is satisfied, then $(M,g)$ is isometric to $(\bar{M},g_{Eucl})$.
\end{thm}
We have the main rigidity theorem in dimension $4$ for negative curvature setting.
\begin{thm}\label{main rigidity thm :hyp}
    Suppose $(M^4,g)$ is a compact Riemannian manifold with scalar curvature lower bound $R g \ge -(n+2)(n+1)g_{\mathbb{H}^{n+2}}$, which is diffeomorphic to a hyperbolic convex ball $\bar{M}^{n+2} \subset \mathbb{H}^{n+2}$. Furthermore, we assume that the shifted mean curvature comparison $(H + 3c \cos \bar{\theta}_{1})^{2} g \ge (H_{\mathbb{H}^{4}} + 3 \cos  \bar{\theta}_{1})^{2} g_{\mathbb{H}^{4}}$ is satisfied and $H + 3c \cos \bar{\theta}_{1} \ge 0$, then $(M,g)$ is isometric to $(\bar{M},g_{Hyp})$.
\end{thm}As a direct corollary of Theorem \ref{main rigidity thm}, we obtain the generalization of Theorem \ref{blueellipseineq} (Gromov's $\fillellipse$-inequality) in dimension $4$ as follows, which removes a topological assumption from the original inequality.
\begin{cor}
    Let $V$ be a Riemannian manifold with boundary which has positive horospherically shifted mean curvature $H_{\partial V} >0$, let $\underline{V} \subset \mathbb{R}^{4}$ be a convex domain with smooth boundary and let $f: \partial V \rightarrow \partial \underline{V}$ is a diffeomorphism. If $R_{V} \ge 0$, then the differential of $f$ cannot be everywhere strictly smaller than the ratio of the mean curvatures of the two boundaries: there exists a point $v \in \partial V$, such that
    \begin{align*}
    ||df(v)|| \ge \frac{H_{\partial V}(f(v))}{H_{\partial \underline{V}}(v)}.
    \end{align*}
\end{cor}
Moreover, we give a hyperbolic version of Gromov's $\fillellipse$-inequality, which is a traced norm version of Gromov's conjectured comparison and rigidity of hyperbolic convex domains in \cite[P. 109, Section 3.5]{gromov2021lecturesscalarcurvature} as follows. 
\begin{cor}
    Let $V$ be a Riemannian manifold with boundary which has positive boundary mean curvature $H_{\partial V} + 3c \cos \bar{\theta}_{1} >0$, let $\underline{V} \subset \mathbb{H}^{4}$ be a convex domain (in a Euclidean coordinate) with smooth boundary and let $f: V \rightarrow \underline{V}$ be a diffeomorphism. If $R_{V} g_{V}\ge -12 g_{\mathbb{H}^{4}}$, then the differential of $f$ cannot be everywhere strictly smaller than the ratio of the shifted mean curvatures of the two boundaries: there exists a point $v \in \partial V$, such that
    \begin{align*}
    ||df(v)|| \ge \frac{H_{\partial V}(f(v)) + 3c \cos \bar{\theta}_{1}}{H_{\partial \underline{V}}(v) + 3 \cos \bar{\theta}_{1}}.
    \end{align*}
\end{cor}
\subsection{Outline of the proof} We outline the proof of main results. The result mainly divides into the proof of the slicing theorem (Theorem \ref{main slicing thm}) finding a nontrivial minimizer for appropriately perturbed functionals and the rigidity theorem (Theorem \ref{main rigidity in general dimension}) deducing global rigidity from local variational information of a two-dimensional bottom capillary minimal slice with its stability analysis.

The first key ingredient is the local adjusted mean curvature estimate (Proposition \ref{prop:local estimates}) between adjusted mean curvature of $\partial M$ and the trace form of the second fundamental form $\RN{2}$ on the boundary $\partial \bar{M}$ of an Euclidean and hyperbolic domain. To handle both cases simultaneously, we use the terminology of the adjusted mean curvature as follows:
\[
\hat{H} = \begin{cases} &H \quad \text{for Euclidean comparison} \\
& (H + c(n+1) \cos \bar{\theta}_{1}) \quad \text{for Hyperbolic comparison}.
\end{cases}
\]
In particular, for any orthonormal frame $\{\tau_i\}_{i=1}^k$ of $T_p \partial M$ with respect to a Riemannian metric $g$ and orthonormal frame $\{\bar{\tau}_i\}_{i=1}^k$ of $T_p \partial \bar{M}$, the following inequality holds:
\begin{align}\label{eq: local estimate intro}
    \hat{H}_{\partial M}\geq \sum_{i=1}^k\RN{2}(\tau_i,\bar{\tau}_i),
\end{align}
This is a generalization of the local curvature estimate in dimension $3$ in Ko-Yao \cite{koyao} (similar trace type quantity and estimate appear in Bär-Brendle-Chow-Hanke \cite{bar2023rigidity} and Wang-Wang-Zhou \cite{wang2024scalar}). The estimate follows from the boundary mean curvature comparison $\hat{H}_g^2g|_{\partial M}\geq \hat{H}_0^2g_{0}|_{\partial \bar{M}}$ for principal curvature directions and is extended to general directions by rotations. In particular, the equality condition aligns $\tau_{i}$ to $\bar{\tau}_{i}$ for all $i$. We remark that key estimates \ref{eq: local estimate intro} in the proof of the main theorem are equivalent to the trace comparison condition used in Bär-Brendle-Chow-Hanke \cite{bar2023rigidity}. 

The main part of the proof of the main slicing theorem (Theorem \ref{main slicing thm}) is devoted to obtaining a winding number in the bottom slice and then obtaining the existence of nontrivial minimizers. More precisely, we use the generalization of the local curvature estimates and an inductive procedure to prove that the boundary integral term which comes from the stability inequality is bounded below by a winding number. 

In dimension $4$, we prove the existence of the minimal slicing in Euclidean case and CMC and minimal slicing in the hyperbolic case. In each slicing, the construction of a graphical foliation which serves as a barrier which has nonnegative mean curvature and contact angle which is larger than the prescribed angle functions $\{ \bar{\theta}_{i} \}$. In particular, the construction of the foliation in the first slicing is the direct generalization of the construction in dimension $3$ in Ko-Yao \cite{koyao}, where the nonnegative (adjusted) mean curvature of leaves follow from the mean curvature comparison between the boundaries in Riemannian and Euclidean (hyperbolic) metric. However, we do not have comparison (weighted) mean curvature at points of the slice where the direction of normal vector of Riemannian and Euclidean(hyperbolic) slicing do not agree. In other words, the normal vector of $\partial \Sigma_{1}$ in $\partial M$ does not agree with the normal vector of $\partial \bar{\Sigma}_{1}$ in $\partial \bar{M}$ in general. To overcome lack of comparison, we pick a point $p_{2}$ on $\partial \Sigma_{1}$ such that the $x_{4}$-coordinate of $p_{1}$ is maximum over $\partial \Sigma_{1}$ and take the axis $x_{3}$ to achieve the maximum coordinate at $p_{1}$. With this procedure, we can identify $T_{p_{1}} \partial \Sigma_{1}$ and $T_{p_{1}} \partial \bar{\Sigma}_{1}$ and apply the strong version of the local curvature estimate, and obtain the desired foliation on each slicing with a specific coordinate. 

The proof of the rigidity statement (Theorem \ref{main rigidity in general dimension}) separates into two parts: extracting the infinitesimal rigidity condition on slices from the stability inequality, and proving the global rigidity from the infinitesimal rigidity. We obtain the infinitesimal rigidity from the stability analysis of the (weighted) capillary area functional. Assuming the existence of a codimension $n$ stable minimal slicing, we can modify the second variation formula by slicing identities and obtain the estimate on the bottom slice $\Sigma_{n}^{2}$ (Proposition \ref{prop: n-th slicing}). Another important ingredients are angle identities (Proposition \ref{prop: Angle Identities}). The identities transform the first derivative of the angle function on the boundary by the second fundamental form on the boundary of the Euclidean domain. Combining the angle identity and the local estimate (\ref{eq: local estimate intro}) with Proposition \ref{prop: n-th slicing}, we obtain the global comparison and rigidity by the winding number argument as in dimension $3$.

We prove the global rigidity using the infinitesimal rigidity by proving the local splitting theorem. Assuming the existence of a flat nontrivial minimizer, we adopt the construction of dense set of flat slices on small neighborhood of the minimizer developed in Carlotto-Chodosh-Eichmair \cite{carlotto2016effective}, relies on the conformal deformation of the metric near the minimizer. We also generalize this to the hyperbolic setting (See Appendix A). Since the argument does not depend on the dimension, we can apply the construction as far as the regularity of the minimizer is guaranteed. In dimension 4, minimizers of the capillary area functional are smooth and we can prove the curvature estimate for area-minimizing capillary minimal hypersurfaces by Chodosh-Edelen-Li \cite{chodosh2024improved}, which gives the compactness of minimizing hypersurfaces.

\subsection{Discussion in higher dimensions} The obstructions to prove the rigidity theorem (Theorem \ref{main rigidity thm}) in dimensions larger than or equal to $5$ are twofold. One is the lack of regularity of minimizer of the capillary area functional and another is that it is hard to find an extrinsic third angle functional that would guarantee the existence of a nontrivial minimizer.

The best full regularity of the capillary minimizer known to authors is codimension $4$ singularity in \cite{chodosh2024improved}, and this may preclude us from the stability analysis of minimizers (See \cite{firester2026area} and \cite{pacati2025some} for partial regularities in higher dimensions). Since the possible singularities of capillary minimizers are isolated singularities in dimension $5$, the regularity issue may be diverted by analyzing stability inequality outside those singular sets (See \cite{he2025positive} and \cite{cecchini2024positive} for such analysis on point singularities).

Another significant issue on iterative capillary minimal slicings is to choice angle functionals for which one can still obtain a nontrivial minimizer. We obtain the existence of a nontrivial minimizer by constructing a barrier which achieves a strictly negative capillary area functional. An eligible foliation can be achieved by a large contact angle with nonnegative mean curvature on leaves. For the first slicing, such a foliation can be constructed in a small neighborhood of every point. For the second cutting, a foliation can be obtained at a point where the tangent planes of the codimension $2$ slices in the $g$ metric and the Euclidean metric agree for the second cutting by maximum principle type arguments. However, we cannot guarantee the existence of such a point for the third slicing. Without the coincidence of two tangent planes $T_{p_{3}} \partial \Sigma_{3}$ and $T_{p_{3}} \partial \bar{\Sigma}_{3}$, We have neither a mean curvature comparison derived from the local estimate (\ref{eq: local estimate intro}) nor a contact angle comparison. 

In terms of a contact angle comparison, a local foliation in a Euclidean metric is necessary to compare contact angles with one in a Riemannian metric. However, we may not have a local foliation without a fixation of two tangent planes since we do not have a control of contact angle function $\bar{\theta}_{3}$ on different slices. 
\subsection{Organization} The organization of this paper is as follows. In Section $2$, we introduce the setup of weighted capillary slicing, derive stability analysis, and prove angle identities. In Section $3$, we prove the local comparison estimate. In Section $4$, we prove the existence of non-trivial minimizers for capillary minimal slicings in dimension $4$. In Section $5$, we prove the rigidity statement by showing the local splitting theorem.
\subsection{Acknowledgements}
D.K. would like to appreciate Tobias Holck Colding, William Minicozzi, and Daniel Ketover for their encouragement and support and X.Y. would like to thank Sun-Yung Alice Chang and Xin Zhou for their encouragement and support. We are grateful to Gioacchino Antonelli, Christine Breiner, Simone Cecchini, Otis Chodosh, Chao Li, Jian Wang and Chao Xia for their interest in our work. D.K. thanks to Raphael Tsiamis for inspiring conversations related to this work. 
\section{Preliminaries}
In this section, we introduce the capillary angle function and weighted capillary slicing which we will use for both Euclidean and hyperbolic settings. Based on the first and second variation formula and first eigenfunction of stability operator of (weighted) capillary area functionals, we derive the curvature estimate on the two dimensional bottom slice. We also prove the angle identity. We formulate our settings on general $(n+2)$-dimensional Euclidean and hyperbolic domains in the last section.
\subsection{Set-up of weighted capillary slicing}

Suppose $\bar{M}^{n+2}$ is a strictly convex domain in $\mathbb R^{n+2}$ or $\mathbb H^{n+2}$. For the hyperbolic space, we take an upper half-space model and the metric is defined by $g_{hyp} = (dx_{1}^{2} + ... +dx_{n+2}^{2})/x_{n+2}^{2}$. Note that Euclidean (half-)space and hyperbolic space are conformally equivalent. We label the $(n+2)$-dimensional Euclidean coordinates as $(x_1,x_2,\cdots, x_{n+2})$ $(x_{n+2}>0)$ both for Euclidean and hyperbolic spaces. $\partial \bar{M}$ is a closed hypersurface with a Euclidean coordinate, we take $p^+_1,p_1^{-}\in\partial \bar{M}$, we have
\[
x_{n+2}(p_1^+)>w(p)>x_{n+2}(p_1^-),\qquad \forall p\in \partial \bar{M}\setminus\{{p_1^+,p_1^-}\}.
\]
Equivalently, we have that the tangent planes $T_{p_1^+}\partial \bar{M}$ and $T_{p_1^-}\partial \bar{M}$ are parallel and their unit normal is $e_{n+2}$ in a Euclidean coordinate. We use Euclidean coordinates to define capillary angle functions. For $x\in\partial \bar{M}$, we define the first capillary angle $\bar{\theta}_1$ as
\begin{align*}
    \bar{\theta}_1(x):=\arccos \langle\bar{X},e_{n+2}\rangle_{g_{Eucl}}.
\end{align*}
Let $\bar{\Sigma}_1^{t_1}=\{x_{n+2}=t_1\}\cap \bar{M}$, we denote $\bar{\eta}_1(x)$ as the unit outer normal of $\partial\bar{\Sigma}_1^{t_1}$ in $\bar{\Sigma}_1^{t_1}$, then $\bar{\eta}_1$ is well-defined for all $x\in \partial \bar{M}\setminus\{p_1^+,p_1^-\}$.We denote the unit co-normal of $\partial\bar{\Sigma}_1^{t_1}$ in $\partial \bar{M} $ as $\bar{\nu}_1$, then it is well-defined for all $x\in \partial \bar{M}\setminus\{p_1^+,p_1^-\}$.

Similarly, we pick the coordinate $x_{n+1}$ and define $\bar{\Sigma}_{2,t_1}^{t_2}=\bar{\Sigma}_1^{t_1}\cap \{x_{n+1}=t_2\}$. For each $x\in\partial\bar{\Sigma}^{t_2}_{2,t_1}$, we define the second capillary angle $\bar{\theta}_2$ as
\begin{align*}
    \bar{\theta}_2(x):=\arccos \langle\bar{\eta}_1(x),e_{n+1}\rangle_{g_{Eucl}}.
\end{align*}
For each $t_1$, we let $p_{2,t_1}^{+}$ ($p_{2,t_1}^-$) be the point on $\partial \bar{\Sigma}_1^{t_1}$ where the maximum (minimal) value of $x_{n+1}$-coordinate is achieved.
Similarly, we define $\bar{\eta}_2$ to be the unit outer normal vector of $\partial\bar{\Sigma}_{2,t_1}^{t_2}$ in $\bar{\Sigma}_{2,t_1}^{t_2}$ and $\bar{\nu}_2$ to be the unit co-normal vector of $\partial\bar{\Sigma}_{2,t_1}^{t_2}$ in $\partial\bar{\Sigma}_1^{t_1}$. Both $\bar{\eta}_2$ and $\bar{\nu}_2$ are well-defined vector fields on $\partial \bar{M}$ except the
Continuing the above process $n$ times, we obtain $k$-th capillary functional $\bar{\theta}_{k}$ for $1\leq k\leq n$.

\subsection{Capillary functionals and variational formulas} \label{sec:capillary functionals}
We are now ready to introduce our capillary functionals and its variation formulas.  The \textit{first capillary functional} is defined as
\begin{align} \label{eqn:first capillary}
\mathcal A_{\bar{\theta}_1}(\Omega):=|\partial^i\Omega|-c(n+1)\text{Vol}_g(\Omega)-\int_{\partial^b \Omega}\cos\bar{\theta}_1 d\mathcal H^{n+1},
\end{align}
where $\Omega\subset M$ is a domain containing $p_1$, and $\partial^i\Omega$ denotes $\partial \Omega\cap \text{Int}(M)$ and $\partial^b\Omega$ denotes $\partial\Omega\cap\partial M$.

We consider the minimization problem of $\mathcal A_{\bar{\theta}_1}$ as follows:
\begin{align}
    \mathcal I_1=\inf\{\mathcal A_{\bar{\theta}_1}(\Omega):\Omega\in\mathcal E_1\},
\end{align}
where $\mathcal E_1$ is a collection of open sets $E_1$ such that $p_1\in E_1$.
Suppose that we obtain a non-trivial minimizer $\Sigma_1$ of $\mathcal A_{\bar{\theta}_1}$ among $\mathcal E_1$. We start with the first and second variation formulas of $\mathcal A_{\bar{\theta}_1}$. 

Suppose $\Sigma=\partial_{rel}\Omega$ separates $M$ into two components. Let $\nu$ be the unit outer normal vector field of $\partial \Sigma$ in $\partial M$, $\eta_1$ be the unit outer normal of $\partial\Sigma_1$ in $\Sigma_1$. 

Let $\Psi(\Sigma_1,t):\Sigma_1\times (-\epsilon,\epsilon)\to M$ be a family of diffeomorphisms such that $\Psi(\Sigma_1,0)=\Sigma_1$ and $\Psi(\partial\Sigma_1,t)\subset \partial M$. Denote $Y=\frac{\partial \Psi}{\partial t}(t,\cdot)$, the vector field that generates $\Psi$. Note that $Y$ lies in the tangent space of $\partial M$. Define $\mathcal A_{\bar{\theta}_1}(t)=\mathcal A_{\bar{\theta}_1}(\Sigma_t)$ and $\varphi=\langle Y,N_1\rangle_g$, where $N_1$ is the unit normal vector of $\Sigma_t$. We use $H_{\Sigma_1}$to denote the mean curvature of $\Sigma$ in $M$. We have the following first and second variation formula of $\mathcal{A}_{\bar{\theta}_1}$.

\begin{prop}\label{prop: first slicing variation}
The first variation formula of $\mathcal A_{\bar{\theta}_1}$ is
\begin{align}
\frac{d}{dt}\bigg|_{t=0}\mathcal A_{\bar{\theta}_1}(\Omega_{t})=-\int_{\Sigma_1}\varphi (H_{\Sigma_1}+c(n+1)) d\mathcal H^{n+1}+\int_{\partial\Sigma_1}\langle Y,\eta-\nu\cos\bar{\theta}_1\rangle_gd \mathcal H^n.
\end{align}

The second variation formula is
\begin{align} \label{secondvariationtheta1}
\frac{d^2}{dt^2}\bigg|_{t=0}\mathcal A_{\bar{\theta}_1}(\Omega_{t})=-\int_{\Sigma_1}\left(\Delta_{\Sigma_1}\varphi+\left(\text{Ric}_{M}(N_1,N_1)+\|A_{\Sigma_1}\|^2\right)\varphi\right)\varphi d\mathcal H^3+\int_{\partial\Sigma_1}\varphi\left(\frac{\partial \varphi}{\partial\eta}-Q_{1}\varphi\right),
\end{align}
where
\begin{align} \label{q1def}
Q_{1}=\frac{1}{\sin \overline{\theta}_1} A_{\partial M} (\nu_1,\nu_1) - \cot \overline{\theta}_1 A_{\Sigma_1}(\eta_1,\eta_1) + \frac{1}{\sin^{2} \overline{\theta}_1} \partial_{\nu_1} \cos \overline{\theta}_1.
\end{align}
\end{prop}

Let $f_1$ be the first eigenfunction of the stability operator of $\mathcal A_{\bar{\theta}_1}$. We define the \textit{second capillary functional} as 
\begin{align}
\mathcal A_{\bar{\theta}_2}(U):=\int_{\partial^i U}f_1d\mathcal H^n-\int_{\partial^b U}\cos\bar{\theta}_2f_1d\mathcal H^n,
\end{align}
where $U$ is an open domain in a connected component of boundary of minimizer $\Sigma_1$ of $\mathcal{A_{\bar{\theta}_1}}$ containing the point $p_2$, where $p_2$ is the point on $\partial \Sigma_1$ that has the maximum $z$-coordinate function. $\partial^i U=\partial U\cap \text{Int}(\Sigma_1)$ and $\partial^b U=\partial U\cap \partial \Sigma$.

Similarly, we consider the minimization problem of $\mathcal A_{\bar{\theta}_2}$ as follows:
\begin{align}
    \mathcal I_2:=\inf\{\mathcal A_{\bar{\theta}_2}(U): U\in\mathcal E_2\},
\end{align}
where $\mathcal E_2$ is the collection of open sets $E_2$ such that $p_2\in E_2$. Suppose $\Sigma_2=\partial^i U$ separates $\Sigma_1$ into two components. Let $\nu_2$ be the unit outer normal vector field of $\partial U$ in $\partial\Sigma_1$, $\eta_2$ be the unit outer normal of $\partial\Sigma_2$ in $\Sigma_2$. By a direct computation, we have the variation formulas for the weighted capillary functional. We follow the computation from \cite{wu2025rigidity}. 
\begin{prop}\label{prop: second slicing variation formula}
    The first variation formula of $\mathcal A_{\bar{\theta}_2}$ is 
\begin{align}\label{eq: bottom first variation}
\frac{d}{dt}\mathcal A_{\bar{\theta}_2}(0)=-\int_{\Sigma_2}\varphi(H_{\Sigma_2}f_1+\langle \nabla f_1,N_2\rangle_g)d\mathcal H^n+\int_{\partial \Sigma_2}\langle Y_2,\eta_2-\nu_2\cos\bar{\theta}_2\rangle_g f_1 d\mathcal H^{n-1}.
\end{align}

The second variation formula is
\begin{align}\label{eq: bottom slice second variation}
\frac{d^2}{dt^2}\mathcal A_{\bar{\theta}_2}(0)=&\int_{\Sigma_2}f_1\varphi\left(-\Delta_{\Sigma_2}\varphi-\left(\|A_{\Sigma_2}\|^2+\text{Ric}_{\Sigma_1}(N_2,N_2)\right)\varphi\right)\, d\mathcal H^n\\
&\nonumber+ \int_{\Sigma_2}f_1\varphi\left(\varphi\nabla_{\Sigma_1}^2\log f_1(N_2,N_2)-\langle \nabla_{\Sigma_2}\log f_1,\nabla_{\Sigma_2}\varphi\rangle_g \right)d\mathcal H^n\\
\nonumber&+\int_{\partial\Sigma_2}\varphi\left(\frac{\partial \varphi}{\partial \eta_2}-Q_2\varphi\right)f_1d \mathcal H^{n-1},
\end{align}
where
\begin{align}
Q_2=&\frac{1}{\sin\bar{\theta}_2}\left(A_{\partial\Sigma_1} \nonumber (\tau_2,\tau_2)+\frac{1}{2}\langle \nabla \log f_1,\eta\rangle_g\right)-\cot\bar{\theta}_2\left(A_{\Sigma_2}(\eta_2,\eta_2)+\frac{1}{2}\langle\nabla \log f_1,N_2\rangle_g\right)\\
&+\frac{1}{\sin^2\bar{\theta}_2}\partial_{\tau_2}\cos\bar{\theta} \nonumber
_2-\frac{1}{2}\langle \nabla \log f_1,\eta_2\rangle_g \\
=&\frac{1}{\sin\bar{\theta}_2}A_{\partial\Sigma_1}(\tau_2,\tau_2)-\cot\bar{\theta}_2A_{\Sigma_2}(\eta_2,\eta_2)
+\frac{1}{\sin^2\bar{\theta}_2}\partial_{\tau_2}\cos\bar{\theta}
_2. \label{simplifiedq2}
\end{align}
\end{prop}
The simplification (\ref{simplifiedq2}) follows from the identity
\begin{align*}
    \eta(x)=\sin\bar{\theta}_2\eta_2(x)+\cos\bar{\theta}_2N_2(x),\quad \forall x\in\partial\Sigma_2,
\end{align*}
where the identity follows from the first variation formula of $\mathcal A_{\bar{\theta}_2}$.

We have the following identity:
\begin{align}
    \Delta_{\Sigma}f_1=\Delta_{\Sigma_2}f_1+\nabla^2f_1(N_2,N_2)+\langle\nabla f_1,N_2\rangle_gH_{\Sigma_2}.
\end{align}

By the first variation formula \eqref{eq: bottom first variation}, we have that
\begin{align*}
    \Delta_{\Sigma}f_1=\Delta_{\Sigma_2}f_1+\nabla^2f_1(N_2,N_2)-f_1^{-1}\langle \nabla f_1,N_2\rangle^{2}_g.
\end{align*}

In dimension $4$, let us now rewrite the second variation formula on $\Sigma_2$ \eqref{eq: bottom slice second variation} as
\begin{align*}
    \frac{d^2}{dt^2}\mathcal A_{\bar{\theta}_2}(0)=&\int_{\Sigma_2}f_1\varphi\left(-\Delta_{\Sigma_2}\varphi-\left(\|A_{\Sigma_2}\|^2+\text{Ric}_{\Sigma}(N_2,N_2)\right)\varphi\right)\\
    & + \int_{\Sigma_2}f_1\varphi\left(\nabla_{\Sigma_{1}}^2\log f_1(N_2,N_2)\varphi-\langle \nabla_{\Sigma_2}\log f_1,\nabla_{\Sigma_2}\varphi\rangle_g \right)d\mathcal H^2\\
    &+\int_{\partial\Sigma_2}\varphi\left(\frac{\partial \varphi}{\partial \eta_2}-Q_2\varphi\right)f_1d \mathcal H^1\\
=&\int_{\Sigma_2}-\nabla_{\Sigma_2}\cdot\left(f_1\varphi\nabla_{\Sigma_2}\varphi\right)+f_1|\nabla_{\Sigma_2}\varphi|^2+\varphi^2\Delta_{\Sigma}f_1\, d\mathcal H^2\\
&\int_{\Sigma_2}-\varphi^2\Delta_{\Sigma_2}f_1-(\|A_{\Sigma_2}\|^2+\text{Ric}_{\Sigma}(N_2,N_2))f_1\varphi^2\, d\mathcal H^2\\
&+\int_{\partial\Sigma_2}\varphi\left(\frac{\partial \varphi}{\partial\eta_2}-Q_2\varphi\right)f_1d\mathcal H^1\\
=&\int_{\Sigma_2}f_1|\nabla \varphi|^2+\langle \nabla_{\Sigma_2}\varphi^2,\nabla_{\Sigma_2}f_1\rangle_g\\
&-\int_{\Sigma_2}\left(\text{Ric}_M(N_1,N_1)+\|A_{\Sigma}\|^2+\text{Ric}_{\Sigma}(N_2,N_2)+\|A_{\Sigma_2}\|^2+\lambda_1\right)f_1\varphi^2\, d\mathcal H^2\\
&+\int_{\partial\Sigma_2}-\varphi^2\frac{\partial f_1}{\partial \eta_2}-Q_2\varphi^2f_1\\
=&\int_{\Sigma_2}f_1|\nabla \varphi|^2+\langle \nabla_{\Sigma_2}\varphi^2,\nabla_{\Sigma_2}f_1\rangle_g\, d\mathcal H^2\\
&-\frac{1}{2}\int_{\Sigma_2}\left(R_g+\|A_{\Sigma}\|^2+\|A_{\Sigma_2}\|^2+H_{\Sigma_1}^2+H_{\Sigma_2}^2+2\lambda_1\right)f_1\varphi^2\, d\mathcal H^2 + \int_{\Sigma_{2}} K_{\Sigma_{2}} f_{1} \varphi^{2} d \mathcal{H}^{2}\\
&-\int_{\partial\Sigma_2}\left(\frac{\partial f_1}{\partial\eta_2}+Q_2f_1\right)\varphi^2d\mathcal H^1,
\end{align*}
where $\lambda_1\geq 0$ is the first eigenvalue of the stability operator of $\Sigma_{1}$, and we applied the Traced Gauss equation twice in the above computations.

Before plugging in a specific test function, we rewrite the $Q_2$ as follows
\begin{align*}
    Q_2=& \frac{1}{\sin\bar{\theta}_2}A_{\partial\Sigma_{1}}(\tau_2,\tau_2)-\cot\bar{\theta}_2A_{\Sigma_2}(\eta_2,\eta_2)+\frac{1}{\sin^2\bar{\theta}_2}\partial_{\tau_2}\cos\bar{\theta}_2\\
    =&-H_{\Sigma_2}\cot\bar{\theta}_2+\frac{H_{\partial\Sigma_{1}}}{\sin\bar{\theta}_2}-\kappa_{\partial\Sigma_2}+\frac{1}{\sin^2\bar{\theta}_2}\partial_{\tau_2}\cos\bar{\theta}_2\\
    =&\cot\bar{\theta}_2\langle \nabla \log f_1, N_2\rangle_g+\frac{H_{\partial\Sigma_{1}}}{\sin\bar{\theta}_2}-\kappa_{\partial\Sigma_2}+\frac{1}{\sin^2\bar{\theta}_2}\partial_{\tau_2}\cos\bar{\theta}_2\\
    \end{align*}

Inductively, one proceeds the $n$-th capillary slicing as follows. Suppose $f_{n-1}$ is the first eigenfunction of the stability operator $\mathcal A_{\bar{\theta}_{n-1}}$, we define the \textit{$n$-th capillary functional} as 
\begin{align}
    \mathcal A_{\bar{\theta}_n}(U_n):=\int_{\Sigma_n}f_1\cdots f_{n-1}-\int_{\partial U_n\cap \partial \Sigma_{n-1}}f_1\cdots f_{n-1}\cos\bar{\theta}_n,
\end{align}
where 
\[
\cos\bar{\theta}_n:=\langle \bar{\eta}_{n-1},e_n\rangle,
\]
$e_n$ is the unit vector of $x_{n}$-direction.
We consider the minimization problem of $\mathcal A_{\bar{\theta}_n}$ as follows:
\begin{align}
    \mathcal I_n:=\inf\{\mathcal A_{\bar{\theta}_n}(U_n):U_n\in \mathcal E_n\},
\end{align}
where $\mathcal E_n$ is the collection of open sets $E_n$ of $\Sigma_{n-1}$ where $p_n\in E_{n}$.

\begin{prop}\label{prop: n-th slicing}
Suppose the bottom capillary slice is $\Sigma_n$, we have that
    \begin{align}
    \int_{\Sigma_n}K_{\Sigma_n}+\int_{\partial\Sigma_n}\kappa_g\geq \int_{\partial\Sigma_n}\frac{1}{\Pi_{i=1}^n\sin\bar{\theta_i}}\left(H_{\partial M}-\sum_{i=1}^n\Pi_{j\leq i}\sin\bar{\theta}_j\partial_{\nu_i}\bar{\theta}_i\right),
    \end{align}
    where $\nu_k$ is the unit normal of $\partial\Sigma_k$ in $\partial\Sigma_{k-1}$, $\eta_k$ is the unit outer normal of $\partial\Sigma_k$ in $\Sigma_k$.
\end{prop}

\begin{proof}
    We denote the bottom slice as $\Sigma_n$, the weighted functional we consider is 
    \[
    \mathcal A_{\bar{\theta}_n}(U_n):=\int_{\Sigma_n}f_1\cdots f_{n-1}-\int_{\partial U_n\cap \partial \Sigma_{n-1}}f_1\cdots f_{n-1}\cos\bar{\theta}_n,
    \]
    where $f_i$ is the first eigen-function of the non-trivial minimizer $\Sigma_i$.

     Let $\rho_{k}=f_1\cdots f_{k-1}$, where $\rho_1=1$, $k\in\{1,\cdots,n\}$.
    By the second variation formula of $\mathcal A_{\bar{\theta}_n}$, and we take $\varphi=\rho_n^{-1/2}$, we have
    \begin{align*}
        \int_{\Sigma_n}-&\frac{3}{4}|\nabla^{\Sigma_n}\log \rho_n|^2+\rho_n^{-1}\Delta_{\Sigma_{n-1}}\rho_n-\left(\|A_{\Sigma_n}\|^2+\text{Ric}_{\Sigma_{n-1}}(N_n,N_n)\right)\\
        &-\int_{\partial\Sigma_n}\left(\rho_n^{-1}\frac{\partial \rho_n}{\partial\eta_n}+Q\right),
    \end{align*}
    where 
    \begin{align*}
    Q&=\frac{1}{\sin\bar{\theta}_n}A_{\partial\Sigma_{n-1}}(\nu_n,\nu_n)-\cot\bar{\theta}_nA_{\Sigma_n}(\eta_n,\eta_n)+\frac{1}{\sin^2\bar{\theta}_n}\partial_{\nu_n}\cos\bar{\theta}_n\\
    &=-H_{\Sigma_n}\cot\bar{\theta}_n+\frac{H_{\partial\Sigma_{n-1}}}{\sin\bar{\theta}_n}-\kappa_{\partial\Sigma_n}+\frac{1}{\sin^2\bar{\theta}_n}\partial_{\nu_n}\cos\bar{\theta}_n\\
    &=\cot\bar{\theta}_n\rho_n^{-1}\langle \nabla \rho_n,N_n\rangle+\frac{H_{\partial\Sigma_{n-1}}}{\sin\bar{\theta}_n}-\kappa_{\partial\Sigma_n}+\frac{1}{\sin^2\bar{\theta}_n}\partial_{\nu_n}\cos\bar{\theta}_n.
    \end{align*}
In the last equality, we used the first variation formula of $\Sigma_n$.

Note that the second variation formula (with a proper test function plugged in) contains two terms: the interior integral term $I$ and the boundary integral term $B$. We first have the following identity on the boundary integral term $B$.

\begin{align*}
    B&=\int_{\partial\Sigma_n}\left(\rho_n^{-1}\frac{\partial \rho_n}{\partial \eta_n}+Q\right)\\
    &=\int_{\partial\Sigma_n} \frac{1}{\sin\bar{\theta}_n}\langle \nabla \log \rho_n,\sin\bar{\theta}_n\eta_n+\cos\bar{\theta}_nN_n\rangle+\frac{H_{\partial\Sigma_{n-1}}}{\sin\bar{\theta}_n}-\kappa_{\partial\Sigma_n}+\frac{1}{\sin^2\bar{\theta}_n}\partial_{\nu_n}\cos\bar{\theta}_n\\
    &=\int_{\partial\Sigma_n}\frac{1}{\sin\bar{\theta}_n}\left(\langle \nabla\log \rho_n,\eta_{n-1}\rangle+H_{\partial \Sigma_{n-1}}\right)-\kappa_{\partial\Sigma_n}+\frac{1}{\sin^2\bar{\theta}_n}\partial_{\nu_n}\cos\bar{\theta}_n\\
    &=\int_{\partial\Sigma_n}\frac{1}{\sin\bar{\theta}_n}\left(\langle \nabla \log \rho_{n-1},\eta_{n-1}\rangle +\langle \nabla \log f_{n-1},\eta_{n-1}\rangle +H_{\partial\Sigma_{n-1}}\right)-\kappa_{\partial\Sigma_n}\\&+\frac{1}{\sin^2\bar{\theta}_n}\partial_{\nu_n}\cos\bar{\theta}_n.
\end{align*}
Since $f_{n-1}$ is the first eigenfunction of the stability operator of $\Sigma_{n-1}$, it satisfies
\begin{align*}
f_{n-1}^{-1}\frac{\partial f_{n-1}}{\partial \eta_{n-1}}&=-H_{\Sigma_{n-1}}\cot\bar{\theta}_{n-1}+\frac{H_{\partial\Sigma_{n-2}}}{\sin\bar{\theta}_{n-1}}-H_{\partial_{\Sigma_{n-1}}}+\frac{1}{\sin^2\bar{\theta}_{n-1}}\partial_{\nu_{n-1}}\cos\bar{\theta}_{n-1}\\
&=\cot\bar{\theta}_n\langle \nabla\log \rho_{n-1},N_{n-1}\rangle+\frac{H_{\partial\Sigma_{n-2}}}{\sin\bar{\theta}_{n-1}}-H_{\partial_{\Sigma_{n-1}}}+\frac{1}{\sin^2\bar{\theta}_{n-1}}\partial_{\nu_{n-1}}\cos\bar{\theta}_{n-1}.
\end{align*}
Plugging in the identity satisfied by $f_{n-1}$, we have
\begin{align*}
    B=&\int_{\partial\Sigma_n}\frac{1}{\sin\bar{\theta}_n}\left(\langle \nabla \log \rho_{n-1},\eta_{n-1}\rangle +\cot\bar{\theta}_{n-1}\langle\nabla\log \rho_{n-1},N_{n-1}\rangle\right)\\
    &+\int_{\partial\Sigma_n}\frac{1}{\sin\bar{\theta}_n}\left(\frac{H_{\partial \Sigma_{n-2}}}{\sin\bar{\theta}_{n-1}}+\frac{1}{\sin^2\bar{\theta}_{n-1}}\partial_{\nu_{n-1}}\cos\bar{\theta}_{n-1}\right)-\int_{\partial\Sigma_n}\kappa_{\partial\Sigma_n}+\frac{1}{\sin^2\bar{\theta}_n}\partial_{\nu_n}\cos\bar{\theta}_n\\
    =&\int_{\partial\Sigma_n}\frac{1}{\sin\bar{\theta}_n\sin\bar{\theta}_{n-1}}\left(\langle \nabla \log \rho_{n-1},\sin\bar{\theta}_{n-1}\eta_{n-1}+\cos\bar{\theta}_{n-1}N_{n-1}\rangle\right) \\
    &+\int_{\partial\Sigma_{n-2}}\frac{1}{\sin\bar{\theta}_n\sin\bar{\theta}_{n-1}}\left(H_{\partial\Sigma_{n-2}}-\partial_{\nu_{n-1}}\bar{\theta}_{n-1}\right)-\int_{\partial\Sigma_n}\kappa_{\partial\Sigma_n}-\frac{1}{\sin\bar{\theta}_n}\partial_{\nu_n}\bar{\theta}_n\\
    =&\int_{\partial\Sigma_n}\frac{1}{\sin\bar{\theta}_n\sin\bar{\theta}_{n-1}}\left(\langle \nabla \log \rho_{n-2},\eta_{n-2}\rangle+\langle \nabla \log f_{n-2},\eta_{n-2}\rangle +H_{\partial\Sigma_{n-2}}\right)\\&-\int_{\partial\Sigma_n}\kappa_{\partial\Sigma_n}
    -\int_{\partial\Sigma_{n}}\left(\frac{1}{\sin\bar{\theta}_n}\partial_{\nu_n}\bar{\theta}_n+\frac{1}{\sin\bar{\theta}_n\sin\bar{\theta}_{n-1}}\partial_{\nu_{n-1}}\bar{\theta}_{n-1}\right)
\end{align*}
Similarly, since $f_{n-2}$ is the first eigenfunction of the stability operator of $\Sigma_{n-2}$, it satisfies that
\begin{align*}
    f_{n-2}^{-1}\frac{\partial f_{n-2}}{\partial\eta_{n-2}}=&-H_{\Sigma_{n-2}}\cot\bar{\theta}_{n-2}+\frac{H_{\partial\Sigma_{n-3}}}{\sin\bar{\theta}_{n-2}}-H_{\partial\Sigma_{n-3}}-\frac{1}{\sin\bar{\theta}_{n-2}}\partial_{\nu_{n-1}}\bar{\theta}_{n-2}.
\end{align*}
Plugging in the identity and continue the above process till the first slice, we have that
\begin{align}
    B=\int_{\partial\Sigma_n}\frac{1}{\Pi_{i=1}^n\sin\bar{\theta_i}}\left(H_{\partial M} + c(n+1) \cos \bar{\theta}_{1}-\sum_{i=1}^n\Pi_{j\leq i}\sin\bar{\theta}_j\partial_{\nu_i}\bar{\theta}_i\right)-\int_{\partial\Sigma_n}\kappa_{\partial\Sigma_n}.
\end{align}

By Schoen-Yau's minimal slicing technique \cite{schoen2017positive}, we have that

\begin{align*}
    I&\leq \int_{\Sigma_n}-\frac{1}{2}\left(R_g+\sum_{i=1}^n\|A_{\Sigma_i}\|^2+\sum_{i=1}^nH_{\Sigma_i}^2\right)+\int_{\Sigma_n}K_{\Sigma_n}.
\end{align*}
Moreover, in the hyperbolic setting we have that
\begin{align}
    \nonumber &\int_{\Sigma_n}-\frac{1} {2}\left(R_g+\sum_{i=1}^n\|A_{\Sigma_i}\|^2+\sum_{i=1}^nH_{\Sigma_i}^2\right)+\int_{\Sigma_n}K_{\Sigma_n} \\ \label{eqn:equalitytracinghyperbolic}
    &\leq \int_{\Sigma_n}-\frac{1}{2}\left(-(n+2)(n+1) \rho^{2}+(n+1) c^{2}+(n+1)^{2}c^{2}\right)+\int_{\Sigma_n}K_{\Sigma_n} \leq \int_{\Sigma_{n}} K _{\Sigma_{n}}.
\end{align}
Hence, in both cases, we have that $I \le \int_{\Sigma_{n}} K_{\Sigma_{n}}$ and the proof is now complete.
\end{proof}
\subsection{Angle Identities}
Suppose $e$ is a vector of $T_p\Sigma$ at the point $p$. Assume further that $p\in \partial\bar{\Sigma}_t$. We prove the following identities which reveal the relation between the second fundamental form in a Euclidean coordinate of $\partial M$ with respect to the Euclidean metric and the directional derivatives of the angle functions. We denote the second fundamental form in a Euclidean coordinate with respect to the Euclidean metric by $\RN{2}_{Eucl}$. 

\begin{prop}[Angle Identities]\label{prop: Angle Identities}
Suppose $e$ is any vector in the tangent space $T_p\partial M$, for any $k\in \{1,2,\cdots,n\}$, we have that
\begin{align}
    \Pi_{i=0}^{k-1}\sin\bar{\theta}_i\partial_{e}\bar{\theta}_k=\RN{2}_{Eucl}(e,\bar{\nu}_k),
\end{align}
    where we define $\sin\bar{\theta}_0=1$.
\end{prop}
\begin{proof}
    Note that
    \begin{align*}
        \partial_{e}\cos\bar{\theta}_k=&\langle \nabla_{e}\bar{\eta}_{k-1},e_k\rangle_{g_{Eucl}}\\
        =&\langle \nabla_e\bar{\eta}_{k-1},\cos\bar{\theta}_k\bar{\eta}_{k-1}-\sin\bar{\theta}_k\bar{\nu}_k\rangle_{g_{Eucl}}\\
        =&-\sin\bar{\theta}_k\langle \nabla_{e}\bar{\eta}_{k-1},\bar{\nu}_k\rangle_{g_{Eucl}},
    \end{align*}
     which gives
     \[
     \partial_{e}\bar{\theta}_k=\langle \nabla_e\bar{\eta}_{k-1},\bar{\nu}_k\rangle_{g_{Eucl}}.
     \]
    
    We prove the following Lemma
    \begin{lem}\label{lem: angle induction}
        For $m\geq 1$, $k\geq m$ we have that
        \begin{align}
            \sin\bar{\theta}_{m-1}\langle \nabla_e\bar{\eta}_{m-1},\bar{\nu}_k\rangle_{g_{Eucl}} =\langle \nabla_{e}\bar{\eta}_{m-2},\bar{\nu}_k\rangle_{g_{Eucl}},
        \end{align}
        we define $\bar{\eta}_0=\bar{X}$.
    \end{lem}
    \begin{proof}
        Note that 
        \begin{align}
            \bar{\eta}_{m-1}=\sin\bar{\theta}_{m-1}\bar{\eta}_{m-2}+\cos\bar{\theta}_{m-1}\bar{\nu}_{m-1}.
        \end{align}
        We look at the term
        \begin{align}
            \langle \nabla_e\bar{\nu}_{m-1},\bar{\nu}_k\rangle_{g_{Eucl}}&=\langle \nabla_e\left(\cos\bar{\theta}_{m-1}\bar{\eta}_{m-1}-\sin\bar{\theta}_{m-1}e_m\right),\bar{\nu}_k\rangle_{g_{Eucl}}\\
           \nonumber &=\cos\bar{\theta}_{m-1}\langle \nabla_e\bar{\eta}_{m-1},\bar{\nu}_k\rangle_{g_{Eucl}}.
        \end{align}
        We have that
        \begin{align*}
            \langle \nabla_e\bar{\eta}_{m-1},\bar{\nu}_k\rangle_{g_{Eucl}}=&\sin\bar{\theta}_{m-1}\langle \nabla_e\bar{\eta}_{m-2},\bar{\nu}_k\rangle_{g_{Eucl}}+\cos^2\bar{\theta}_{m-1}\langle \nabla_{e}\bar{\eta}_{m-1},\bar{\nu}_k\rangle_{g_{Eucl}}.
        \end{align*}
        Moving the $\cos^2\bar{\theta}_{m-1}\langle \nabla_e\bar{\eta},\bar{\nu}_k\rangle_{g_{Eucl}}$ term to the left hand side,the identity follows easily.
    \end{proof}
    With Lemma \ref{lem: angle induction}, we have that
    \begin{align*}
        \Pi_{i=0}^{k-1}\sin\bar{\theta}_i\partial_e\bar{\theta}_k=&\langle \nabla_e\bar{X},\bar{\nu}_k\rangle_{g_{Eucl}}\\
        =&\RN{2}_{Eucl}(e,\bar{\nu}_k).
    \end{align*}
\end{proof}
\subsection{Comparison setups} In this section, we introduce the formulation of curvature comparison settings both on interior and boundary in our applications of capillary slicing to the Euclidean and hyperbolic settings. Based on the comparison settings, we provide a choice of $c$ in the first capillary functional in Section \ref{sec:capillary functionals}.
\subsubsection{Euclidean setting} $(M^{n+2},g)$ is a $(n+2)$-dimensional manifold with boundary with non-negative scalar curvature $R \ge 0$, whose boundary is diffeomorphic to the boundary of a Euclidean convex domain $(\bar{M}^{n+2},g_{Eucl})$. We assume that $\partial M$ satisfies the scaling invariant mean curvature comparison $H_g^2g|_{\partial M}\geq H_0^2g_{E}|_{\partial \bar{M}}$. We take $c=0$ in the first capillary functional (\ref{eqn:first capillary}) and hence a minimizer of the first capillary functional is a capillary minimal surface with a contact angle $\bar{\theta}_{1}$.
\subsubsection{Hyperbolic setting} \label{sec:hyp setting} We consider an upper half-space model as before for a hyperbolic space $\mathbb{H}^{n+2}$. Then a convex body in $\mathbb{H}^{n+2}$ is foliated by horospherical slices with mean curvature $(n+1)$ given by its $x_{n+2}$-level sets. We consider scaling invariant scalar curvature comparison $R g \ge R_{\mathbb{H}^{n+2}} g_{hyp}$, for instance, $R_{\mathbb{H}^{n+2}} = -(n+2)(n+1)$. 

For each point $x \in \Omega$, we define the \emph{minimal stretch} $\rho : \Omega \rightarrow \mathbb{R}_{+}$ to be
\[
\rho^{2}(x) = \min _{|v |_{g} =1} g_{Hyp}(v,v).
\]

Let us take $c = \max_{x \in \Omega} \rho(x)$ and call this by \emph{scaling factor}. We now regard $\hat{g}_{Hyp} = c^{-2} g_{Hyp}$ as a reference domain by scaling. Note that
\[
c x_{n+2}A_{g_{Eucl}} = c(A_{g_{Hyp}} +  \langle \nabla^{g_{Eucl}}  x_{n+2},\, \nu_{g_{Eucl}} \rangle_{g_{Eucl}} I) = c A_{g_{Hyp}} +  c \cos  \bar{\theta}_{1} I.
\]
Hence, we define a \emph{horospherically shifted mean curvature comparison} on boundary as follows.
\[
(H + c(n+1) \cos \bar{\theta}_{1})^{2} g \ge (H_{\mathbb{H}^{n+2}} + (n+1) \cos  \bar{\theta}_{1})^{2} g_{\mathbb{H}^{n+2}}.
\]
\begin{rem}
    In the Euclidean setting, we only require a diffeomorphism between boundaries since the scalar curvature comparison is on scalar functions. However, we impose a requirement on a diffeomorphism on the whole manifold in the hyperbolic setting since our scaled curvature comparison is a tensor inequality with nonzero scalar curvature. 
    Moreover, if one assumes curvature comparison with a mapping version, then some connected sum of small nonnegative scalar curvature manifold and hyperbolic ball may satisfy the scaled scalar curvature comparison, and this violates the rigidity statement. For this reason, on the scalar curvature comparison and rigidity in negative curvature setting, it may be more appropriate to consider a pointwise curvature comparison, rather than rigidity results in the mapping version in the positive scalar curvature settings, for instance, by Llarull \cite{llarull1998sharp} and Listing \cite{listing2010scalar}.
\end{rem}
To unify discussions later, we use the \emph{adjusted mean curvature} $\hat{H}$ as a combined form of mean curvature:
\[
\hat{H} = \begin{cases} &H \quad \text{for Euclidean comparison} \\
& (H + c(n+1) \cos \bar{\theta}_{1}) \quad \text{for Hyperbolic comparison}.
\end{cases}
\]
Note that the $(n+2)$-dimensional hyperbolic space takes $c=1$ and $\hat{H} = H + (n+1) \cos \bar{\theta}_{1}$.
\section{Local estimates} \label{section: local estimates}

$\Sigma$ is a non-trivial minimizer of the capillary functional $\mathcal {A}_{\bar{\theta}_1}$, we have the following local estimates on $\partial\Sigma$. We write the local estimate in the most general form which works for weakly convex domains. We define $F: T_{p}\partial M \rightarrow T_{p}\partial M$ to be an isomorphism which sends $\tau_{i}$ to $\bar{\tau}_{i}$ for $i \in \{1,2,\cdots,n\}$ below and positive curvature space $\RN{2}_{+}$ by $\RN{2}_{+}(p):=\operatorname{span}\{\,v\in T_p\partial M:\ \RN{2}_{Eucl}(v,v)>0\,\}$ which is a subspace of $T_{p}M$.

\begin{prop}\label{prop:local estimates}
For any $p\in \partial\Sigma$, suppose that $\{\tau_i\}_{i=1}^n$ is an orthonormal frame of $T_p \partial M$ with respect to $g$. Assume further that $p\in\partial\bar{\Sigma}_t$ for some $t$, and $\{\bar{\tau}_i\}_{i=1}^n$ is an orthonormal frame of $T_p \partial M$, with respect to $g_{Eucl}$, then
\begin{align}\label{eq: local estimate}
    \hat{H}_{\partial M}\geq \sum_{i=1}^n\RN{2}_{Eucl}(\tau_i,\bar{\tau}_i),
\end{align}
where $\hat{H}_{\partial M}^{2} g|_{F^{\#}\RN{2}_{+}} = \hat{\bar{H}}_{\partial M}^{2} g_{Eucl}|_{\RN{2}_{+}}$ when the equality holds. 

\begin{proof}
   Since the inequality is symmetric, we can rotate $\{\bar{\tau}_i\}_{i=1}^n$ appropriately and take the matrix $(a_{ij})_{1 \le i,j \le n} = (\RN{2}(\bar{\tau}_i,\bar{\tau}_j))_{1 \le i,j \le n}$ to be a diagonal matrix. We rotate $\{\tau_i\}_{i=1}^n$ simultaneously with $\{\bar{\tau}_i\}_{i=1}^n$. More specifically there exists an orthogonal matrix $(O_{ij})_{n\times n}\in SO(n)$ such that $\tau_{i}' = O_{ij}\tau_{j}$ and $\bar{\tau}_{i}' = O_{ij} \bar{\tau}_{j}$ and we replace $\tau_{i}$ and $\bar{\tau}_{i}$ by $\tau_{i}'$ and $\bar{\tau}_{i}'$, respectively. Note that $\{\tau_i\}_{i=1}^n$ is preserved to be an orthonormal set in metric $g$ since a rotation is represented by an orthogonal matrix.
 \\
        Let
    \[
    W_i=\frac{a_{ii}}{\hat{\bar{H}}_{\partial M}}
    \]
    for $i\in\{1,2,\cdots,3\}$ and note that $\sum_{i=1}^nW_i=1$. By the boundary comparison condition, we have
\begin{align}
  \label{threepositivemccomparison} \hat{H}_{\partial M}\geq& \sum_{i=1}^nW_i\hat{\bar{H}}_{\partial M} g_{Eucl}(\tau_i,\bar{\tau}_{i}) \\
  \nonumber &= \sum_{i=1}^na_{ii} g_{Eucl}(\tau_i,\bar{\tau}_{i})\\
    &=\sum_{i=1}^n\RN{2}_{Eucl}(\tau_i,\bar{\tau}_i) \label{lem2.1diag},
    \end{align}
where (\ref{threepositivemccomparison}) gives $\hat{H}_{\partial M}^{2} g|_{F^{\#}\RN{2}_{+}} = \hat{\bar{H}}_{\partial M}^{2} g_{Eucl}|_{\RN{2}_{+}}$ (\ref{lem2.1diag}) follows from the coordinate decomposition and the fact that $(a_{ij})$ is a diagonal matrix.
Also note that
\begin{align*}
    \sum_{i=1}^n\RN{2}_{Eucl}(O_{ij}\tau_j,O_{ik}\bar{\tau}_k)&= \sum_{i=1}^nO_{ij}O_{ik}\RN{2}_{Eucl}(\tau_j,\bar{\tau}_k) \\ &= \delta_{jk} \RN{2}_{Eucl}(\tau_j,\bar{\tau}_k) \\ &= \sum_{i=1}^n\RN{2}_{Eucl}(\tau_i,\bar{\tau}_i).
\end{align*}
Hence the right hand side sum is an invariant over rotations. The equality condition also follows from the rotation.
\end{proof}

\end{prop}

As a corollary, we obtain the key boundary estimate.
\begin{cor}[Key boundary estimate]\label{cor: key estmate}
    For any $p\in\partial\Sigma_n$, we have that
    \begin{align}
       \hat{H}_{\partial M}-\sum_{i=1}^n\Pi_{j\leq i}\sin\bar{\theta}_j\partial_{\nu_i}\bar{\theta}_i\geq \RN{2}_{Eucl}(\tau_1,\bar{\tau}_1).
    \end{align}
\end{cor}
\begin{proof}
    It follows from Lemma \ref{prop: Angle Identities} and Proposition \ref{prop:local estimates}.
\end{proof}
\subsection{Winding  Number}
Since $\bar{M}$ is a convex domain in $\mathbb R^{n+2}$, we parametrize $\partial \bar{M}$ as 
        \begin{align}
            \Psi(\vec{u})=\Psi(u_1,\cdots, u_{n+1}):=(x(\vec{u}),y(\vec{u}),u_2,u_3,\cdots, u_{n+1}).
        \end{align}
  \begin{rem}
      To obtain a winding number, we need to slice the $(n+2)$-dimensional domain $n$-times.
  \end{rem}      
\begin{prop} \label{prop: winding number}
\begin{align}
    \int_{\partial\Sigma_n}\frac{1}{\Pi_{i=1}^n\sin\bar{\theta_i}}\left(\hat{H}_{\partial M}-\sum_{i=1}^n\Pi_{j\leq i}\sin\bar{\theta}_j\partial_{\nu_i}\bar{\theta}_i\right)\geq 2k\pi,
    \end{align}
    for some $k\in\mathbb N_+$.
\end{prop}
\begin{proof}
    By Proposition \ref{prop: Angle Identities} and Proposition \ref{prop:local estimates}, we have that
    \begin{align*}
        \int_{\partial\Sigma_n}\frac{1}{\Pi_{i=1}^n\sin\bar{\theta_i}}\left(\hat{H}_{\partial M}-\sum_{i=1}^n\Pi_{j\leq i}\sin\bar{\theta}_j\partial_{\nu_i}\bar{\theta}_i\right)dl_g\geq \int_{\partial\Sigma_n}\frac{\RN{2}(\tau_1,\bar{\tau}_1)}{\Pi_{i=1}^n\sin\bar{\theta}_i}dl_g
        \end{align*}
        Since $\bar{M}$ is a convex domain in $\mathbb R^n$, we parametrize $\partial \bar{M}$ as 
        \begin{align}
            \Psi(\vec{u})=\Psi(u_1,\cdots, u_{n+1}):=(x(\vec{u}),y(\vec{u}),u_2,u_3,\cdots, u_{n+1}),
        \end{align}
        and $\bar{\tau}_1$ is parallel to $\frac{\partial\Psi}{\partial u_1}$.
        Under this parametrization, we know that $\bar{X}$ is parallel to
        \begin{align*}
           \mathcal{X}_n=\det \left(\begin{matrix}
                \hat{i}_1 & \hat{i}_2 & \hat{i}_3 &\cdots &\hat{i}_{n+2}\\
                x_{u_1} & y_{u_1} & 0 &\cdots & 0\\
                x_{u_2} & y_{u_2} & 1 & \cdots & 0\\
                x_{u_3} &y_{u_3} &0 &\cdots & 0\\
                \vdots  & \vdots & \vdots &\ddots & \vdots\\
                x_{u_{n+1}} & y_{u_{n+1}} & 0 &\cdots & 1 
            \end{matrix}\right)
        \end{align*}
        Similarly, one can define $\mathcal{X}_k$ for $k\in\{2,3,\cdots, n\}$, we use this definition to compute the determinant in a clear way.
        \begin{lem}
            \begin{align}\label{normalvecinduction1}
                \mathcal{X}_{n}=\mathcal{X}_{n-1}+c_{n+2}\hat{i}_{n+2}.
            \end{align}
        \end{lem}
        \begin{proof}
            For any $\hat{i}_{m}$ component of $\mathcal{X}_n$ when $m\in\{1,2,\cdots, n\}$, the coefficient is 
            \begin{align*}
                \det\left(\begin{matrix}
                    M_{n-1,m} & \vec{0}\\
                    \vec{v}_n & 1
                \end{matrix}\right),
            \end{align*}
            where $\vec{0}$ is a $(n-1)\times 1$ column whose entries are all zero,
            \[
            \vec{v}_n=(x_{u_{n+1}}, y_{u_{n+1}}, 0,\cdots, 0),
            \]
            and $\det M_{n-1,m}$ is the coefficient of $\hat{i}_m$ in $\mathcal{X}_{n-1}$. The equation \eqref{normalvecinduction1} follows.
        \end{proof}
        Our next step is to compute $c_n$.

       \begin{lem}
           \begin{align}
               c_{n+2}=x_{u_1}y_{u_{n+1}}-x_{u_{n+1}}y_{u_1}.
           \end{align}
       \end{lem}
       \begin{proof}
           Note that 
           \begin{align*}
               &c_{n+2}=\\
               &(-1)^{n-1}\left\{(-1)^{n-2}x_{u_{n+1}}\det\left(\begin{matrix}
                   y_{u_1} & 0 & 0 &\cdots & 0\\
                   y_{u_2} & 1 & 0 & \cdots & 0\\
                   y_{u_3} & 0 & 1 & \cdots & 0\\
                   \vdots & \vdots & \vdots & \ddots & \vdots \\
                   y_{u_{n}} & 0  & 0 &\cdots& 1
               \end{matrix}\right)+(-1)^{n-1}y_{u_{n+1}}\det\left(\begin{matrix}
                   x_{u_1} & 0 & 0 &\cdots & 0\\
                   x_{u_2} & 1 & 0 & \cdots & 0\\
                   x_{u_3} & 0 & 1 & \cdots & 0\\
                   \vdots & \vdots & \vdots & \ddots & \vdots \\
                   x_{u_{n}} & 0  & 0 &\cdots& 1
               \end{matrix}\right)\right\}\\
               =&x_{u_1}y_{u_{n+1}}-x_{u_{n+1}}y_{u_1}.
           \end{align*}
       \end{proof}
    We then have that
    \begin{align}
        \mathcal{X}_n=&\mathcal{X}_1+\sum_{m=3}^{n+2}\left(x_{u_1}y_{u_{m-1}}-x_{u_{m-1}}y_{u_1}\right)\hat{i}_m\\
        \nonumber=&y_{u_1}\hat{i}_1-x_{u_1}\hat{i}_2+\sum_{m=3}^{n+2}\left(x_{u_1}y_{u_{m-1}}-x_{u_{m-1}}y_{u_1}\right)\hat{i}_m.    \end{align}

    Since $\bar{X}$ is orthogonal to $\bar{\tau}_1$ in the Euclidean metric, we have
    \begin{align*}
        g_{Eucl}(\nabla_{\gamma'(t)}\bar{X},\bar{\tau}_1)&=\frac{1}{|\mathcal{X}_n|}g_{Eucl}(\frac{d}{dt}\mathcal{X}_n,\bar{\tau}_1)\\
        &=\frac{\sum_{i=1}^{n+1}x_{u_1}y_{u_1u_i}u_i'-\sum_{i=1}^{n+1}y_{u_1}x_{u_1u_{i}}u_i'}{|\mathcal{X}_n|}\\
        &=\frac{\sum_{i=1}^{n+1}x_{u_1}y_{u_1u_i}u_i'-\sum_{i=1}^{n+1}y_{u_1}x_{u_1u_{i}}u_i'}{\sqrt{y_{u_1}^2+x_{u_1}^2+\sum_{m=3}^{n+2}(x_{u_1}y_{u_{m-1}}-x_{u_{m-1}}y_{u_1})^2}\sqrt{x^2_{u_1}+y^2_{u_1}}}.
    \end{align*}

\begin{lem}
    For $m\in\{1,2\cdots, n\}$, we have that
    \begin{align}\label{eq:angle induction}
        \Pi_{i=1}^m\sin\bar{\theta}_i=\frac{|\mathcal{X}_{n-m}|}{|\mathcal{X}_{n}|}.
    \end{align}
\end{lem}
\begin{proof}

We prove by induction. For $m=1$, we have that
\begin{align*}
    \cos\bar{\theta}_1=&g_{Eucl}(\bar{X},e_n)\\
                      =&\frac{1}{|\mathcal{X}_n|}g_{Eucl}(\mathcal{X}_n,e_n)\\
                      =&\frac{c_n}{|\mathcal{X}_n|}.
\end{align*}

Now we have that
\begin{align*}
    \sin\bar{\theta}_1=&\sqrt{1-\cos^2\bar{\theta}_1}\\
                      =&\sqrt{1-\frac{c_n^2}{|\mathcal{X}_n|^2}}\\
                      =&\sqrt{1-\frac{|\mathcal{X}_n|^2-|\mathcal{X}_{n-1}|^2}{|\mathcal{X}_n|^2}}\\
                      =&\frac{|\mathcal{X}_{n-1}|}{|\mathcal{X}_n|},
\end{align*}
where we used \eqref{normalvecinduction1}.

We verified \eqref{eq:angle induction} for $m=1$. Assume that $\eqref{eq:angle induction}$ is true for $1\leq i\leq m$, we consider the case when $i=m+1$.

Since for any $m\in\{1,2,\cdots, n-1\}$, we have that
\begin{align*}
    \bar{\eta}_{m}=\frac{1}{\sin\bar{\theta}_m}\left(\bar{\eta}_{m-1}-\cos\bar{\theta}_{m}e_{n+2-m}\right).
\end{align*}

By definition 
\begin{align*}
    \cos\bar{\theta}_{m+1}=&g_E(\bar{\eta}_{m},e_{n+1-m})\\
                          =&\frac{1}{\sin\bar{\theta}_m}g_E(\bar{\eta}_{m-1},e_{n+1-m})
\end{align*}
Observe that 
\begin{align*}
    \bar{\eta}_{m-1}=\frac{1}{\sin\bar{\theta}_{m-1}}\left(\bar{\eta}_{m-2}-\cos\bar{\theta}_{m-1}e_{n+3-m}\right),
\end{align*}
where the $e_{n+2-m}$ component is perpendicular to $e_{n+1-m}$, and thus we have that
\begin{align*}
    \cos\bar{\theta}_{m+1}=\frac{1}{\sin\bar{\theta}_{m}\sin\bar{\theta}_{m-1}}g_E(\bar{\eta}_{m-2},e_{n+1-m}).
\end{align*}
The above procedure proceeds, and we finally have that
\begin{align*}
    \cos\bar{\theta}_{m+1}=&\frac{g_{Eucl}(\bar{X},e_{n+1-m})}{\Pi_{i=1}^{m}\sin\bar{\theta}_i}\\
    =&\frac{c_{n+2-m}}{|\mathcal{X}_n|\Pi_{i=1}^m\sin\bar{\theta}_i}\\
    =&\frac{c_{n+2-m}}{|\mathcal{X}_{n-m}|},
\end{align*}
where we used the induction hypothesis in the last step.
We then have that
\begin{align*}
    \sin\bar{\theta}_{m+1}=&\sqrt{1-\frac{c_{n+2-m}^2}{|\mathcal{X}_{n-m}|^2}}\\
    =&\frac{|\mathcal{X}_{n-m-1}|}{|\mathcal{X}_{n-m}|},
\end{align*}
where we used \eqref{normalvecinduction1}.
Now we have 
\[
\Pi_{i=1}^{m+1}\sin\bar{\theta}_i=\frac{|\mathcal{X}_{n-m-1}|}{|\mathcal{X}_n|},
\]
the induction completes.
\end{proof}
Now, we have that
\begin{align*}
    \int_{\partial\Sigma_n}\frac{\RN{2}_{Eucl}(\tau_1,\bar{\tau}_1)}{\Pi_{i=1}^{n}\sin\bar{\theta}_i}=&\int_{\partial\Sigma_n}
    \frac{\sum_{i=1}^{n+1}\left(x_{u_1}y_{u_1u_i}u_i'(t)-y_{u_1}x_{u_1u_{i}}u_i'(t)\right)}{|\mathcal{X}_n|\sqrt{x_{u_1}^2+y_{u_1}^2}}\Pi_{i=1}^{n}\sin\bar{\theta}_1dt\\
    =&\int_{\partial\Sigma_n}\frac{\sum_{i=1}^{n+1}\left(x_{u_1}y_{u_1u_i}-y_{u_1}x_{u_1u_i}\right)u_i'(t)}{x_{u_1}^2+y_{u_1}^2}dt\\
    =&\int_{\partial\Sigma_n}\frac{d}{dt}\arctan\left(\frac{y_{u_1}}{x_{u_1}}\right)dt\\
    =&2k\pi,
\end{align*}
for some $k\geq 1$.
\end{proof}
\section{Non-trivial minimizer}
In this section, we prove that there exists a non-trivial minimizer for $\{\mathcal{A}_{\bar{\theta}_{i}}\}_{i=1}^{2}$ in $M^{4}$ both in the Euclidean and hyperbolic settings. To prove the existence of a non-trivial minimizer, we construct a foliation near the point $p_i$ for $i=1,2$ then prove that the foliations we construct act as good barriers. For notational convenience, we define normalized adjusted mean curvature $\tilde{H}$ as follows.
\[
\tilde{H} = \begin{cases} &\hat{H} \quad \text{for Euclidean comparison} \\
& \left(\frac{1}{c x_{4}}\right)\hat{H} \quad \text{for Hyperbolic comparison}.
\end{cases}
\]
We also write the conformal factor for each comparison by
\[
\lambda_{*} := \begin{cases} &1 \quad \text{for Euclidean comparison} \\
& cx_{4} \quad \text{for Hyperbolic comparison}.
\end{cases}
\]
\subsection{First Slicing}
Let $\{\frac{\partial}{\partial x_i}\}_{i=1}^{3}$ be an orthonormal basis of $T_{p_1}\partial M$ with respect to the Euclidean metric. Note that $\{ x_{4} \frac{\partial}{\partial x_i}\}_{i=1}^{3}$ is an orthonormal basis of $T_{p_1}\partial M$ with respect to the hyperbolic metric.

By the implicit function theorem, we can write $\partial M$ in a neighborhood of $p_1=\vec{0}$ as
\begin{align} \label{boundarygraph}
    \partial M\cap N_{p_1}= \{(x_1,x_2,x_{3},-c_{ij}(x)x_ix_j+O(|x|^3))\},
\end{align}
where $N_{p_1}$ is a small neighborhood of $p_1$ in $\mathbb R^{4}$, and $C(x)=(c_{ij}(x))$ can be written as
\[
c_{ij}(x)=\frac{1}{2}\RN{2}_{Eucl}\left(\frac{\partial}{\partial x_i},\frac{\partial}{\partial x_j}\right)+O(|x|^4),
\]
where $c_{ij}$ is a positive definite matrix whose entries are all non-zero when $|x|\neq 0$.

\begin{rem}
    The reason to define our matrix $C$ depending on the point $x$ is to guarantee that all the entries of $C$ are non-zero and we can apply the following linear algebra techniques. 
\end{rem}

Note that $(x_1,x_2,x_3,x_4)$ is also a local coordinate of $M^{4}$ in a small neighborhood of $p_1$, by rotating the coordinate accordingly, we can write the metric $g$  at the point $p_1$ as $g(p_1)= (a_{ij})_{4\times 4}$ and the $3 \times 3$ principal matrix of $g(p)$ to be $g(p_{1})|_{\partial M}$ and a diagonal matrix i.e. $a_{ij}=0$ for $i \neq j$ and $1 \le i, j \le 3$.

We can also assume that $g= g_{0}+th+O(t^{2})$ where $h$ is a bounded symmetric $2$-tensor. We construct the following graphical foliation near $p_1$.
\begin{align} \label{foliationgraph}
    \Sigma_{s,t}:=\{(tx_1,tx_2,tx_{3}, t^2\{c_{ij}\left(b_{ij}(1+s)-1\right) x_{i} x_{j}+O(t^4)\}): \vec{x}\in E_s\},
\end{align}
where 
\[
E_s:=\{\vec{x}\in\mathbb R^{3}: c_{ij}b_{ij}x_ix_j<(1+s)^{-1}\},
\]
and $B=(b_{ij})$ is a symmetric $3\times 3$ matrix defined as follows. We abbreviate the notation when $s=0$: $\Sigma_{t} := \Sigma_{s,t}$.

Let $D= \frac{1}{a^{44}} g(p_{1})|_{\partial M}^{-1}$ and this follows $D_{ii} = \frac{1}{a_{ii} a^{44}}$ where $D$ is a diagonal matrix. Since $C$ is a positive definite matrix, we define 
    \[
    S=\left(D^{\frac{1}{2}}C^TCD^{\frac{1}{2}}\right)^{\frac{1}{2}},
    \]
    and
    \[
    V:=D^{-1/2}SD^{-1/2}.
    \]
    It is straightforward to notice that both $S$ and $V$ are symmetric.

    Finally, we define $B=(b_{ij})$ as
    \[
    b_{ij}=\frac{V_{ij}}{c_{ij}}.
    \]
   By definition, $B$ is a symmetric $3\times 3$ matrix. 
   \begin{rem}
       In fact, all the matrices defined above depend on the coordinate to a higher order. Since all our estimates allow higher order errors, we do not emphasize the dependence.
   \end{rem} 
\begin{prop} [Contact Angle Comparison]\label{firstcontactanglecomp}
    Suppose $\theta^{g}_{s,t}(x)$ is the contact angle of $\Sigma_{s,t}$ and $\partial M$ at $x\in \partial\Sigma_{s,t}$ with respect to the metric $g$; and $\bar{\theta}$ is a prescribed angle function with respect to the Euclidean metric. 
    We have
    \begin{align}
        \cos\theta^g_{s,t}(x)=\cos\bar{\theta}(x)-4st^2\sum_i\left(\frac{(\sum_jc_{ij}b_{ij}x_{j})^2}{a_{ii}a^{44}}\right)+O(s^2)+O(t^3).
    \end{align}
\end{prop}
\begin{proof}
By straightforward computations as in dimension $3$ \cite{koyao}, we have 
    \begin{align*}
        \cos\theta^g_{s,t}=1-2t^2\sum_{i}\frac{(c_{ij}b_{ij}(1+s)x_{j})^2}{a_{ii}a^{44}}+O(t^3),
    \end{align*}
    and
    \begin{align*}
        \cos\bar{\theta}(x)=1-2t^2\sum_{i}(c_{ij}x_{j})^2+O(t^3).
    \end{align*}
    By the definition of $B$, we have the following identities:
    \begin{align*}
    \sum_{i}\frac{c_{ik}c_{i l}b_{ik}b_{il}}{a_{ii}a^{44}}&=\sum_{i}\frac{V_{ik}V_{il}}{a_{ii}a^{44}}\\
    &=(VDV)_{kl}\\
    &=(D^{-1/2}S^2D^{-1/2})_{kl}\\
    &=(C^TC)_{kl}\\
    &=\sum_{i}c_{ik}c_{il},
    \end{align*}
    for any $k,l\in \{1,2,3\}$.
    The desired estimates then follow.
\end{proof}

We compute the mean curvature formula of parametrized hypersurfaces.
\begin{lem}[Mean curvature formula of graphs] \label{graphmcformula} Suppose $\Sigma_{t}$ is a family of hypersurfaces which is parametrized as
\[
\Sigma_{t} = \{(tx_1,tx_2,tx_{3}, t^2\{c_{ij} x_{i} x_{j}+O(t^2)\}): \vec{x}\in E_s\},
\]
where $\Sigma_{t}$ is supported in a small neighborhood of $p_{1}$. Then for small $t>0$, the mean curvature $H_{\Sigma_{t}}$ under $g(p_{1})$ is
\begin{equation*}
    H_{\Sigma_{t}}^{g(p_{1})}(x) = - \frac{1}{ \sqrt{a^{44}}} \bigg( \sum_{i} \frac{2c_{ii}}{a_{ii}}\bigg) + O(t).
\end{equation*}
\begin{proof}
The formula follows directly from the proof of Lemma 5.4 in \cite{koyao}.
\end{proof}
\end{lem}
By applying Lemma \ref{graphmcformula}, we obtain the following mean curvature difference directly. We denote the graph representation of $\Sigma_{t}$ by $(tx_1,tx_2,tx_{3}, G(tx_{i}))$.
\begin{prop} \label{mccomparisong} On an intersection point $(tx_1,tx_2,tx_{3}, G(tx_{i})) \in \partial M \cap \Sigma_{t}$, the normalized adjusted mean curvature difference $\tilde{H}_{\Sigma_{t}}^{g(p_{1})} - \tilde{H}_{\partial M}^{g(p_{1})}$ is
\begin{equation}\label{mcdiff}
  \tilde{H}_{\Sigma_{t}}^{g(p_{1})} - \tilde{H}_{\partial M}^{g(p_{1})} = -\frac{2}{\lambda_{*}} \sqrt{a^{44}} \Tr (S)  +O(t).  
\end{equation}
\begin{proof} Notice that $1- \cos \bar{\theta}_{1} = O(t^{2})$ near $p_{1}$. Hence, by applying Lemma \ref{graphmcformula} to the graphical parametrizations of $\partial M$ and $\Sigma_{s,t}$ in (\ref{boundarygraph}) and (\ref{foliationgraph}), respectively, we obtain the following:
\begin{align}
\nonumber \tilde{H}_{\Sigma_{t}}^{g(p_{1})} - \tilde{H}_{\partial M}^{g(p_{1})}  &=  -\frac{2}{\lambda_{*} \sqrt{a^{44}}} \bigg( \sum_{i} \frac{c_{ii}b_{ii}}{a_{ii}}\bigg) + O(t) \\ &= -\frac{2}{ \lambda_{*}\sqrt{a^{44}}} \bigg( \sum_{i} \frac{V_{ii}}{a_{ii}}\bigg) + O(t) \label{vdef},
\end{align}
where (\ref{vdef}) is obtained by the definition of the matrix $B$. By the definition of $V$, we have
\begin{align}
 \nonumber V_{ii} &= D^{-\frac{1}{2}}_{ik} S_{kl} D^{-\frac{1}{2}}_{li} \\ \nonumber &= D^{-\frac{1}{2}}_{ii} S_{ii} D^{-\frac{1}{2}}_{ii} \\ &= S_{ii}a_{ii}a^{44} \label{Vcomp}
\end{align}
for $i=1,2,3$. By plugging (\ref{Vcomp}) into (\ref{vdef}), we obtain (\ref{mcdiff}).
\end{proof}
\end{prop}
We now compute the difference of the adjusted mean curvature difference of leaves and of boundaries in between the metric $g$ and the reference metric. Throughout this section, we write $g_{*} = g_{Eucl} \text{ or } g_{Hyp}$ for the reference metric. 
\begin{cor} The asymptotic of the normalized adjusted mean curvature difference at $(tx_1,tx_2,tx_{3}, G(tx_{i})) \in \partial M \cap \Sigma_{t}$ is
\begin{align} \label{mcgeuclcomparison}
    \tilde{H}_{\Sigma_{t}}^{g} &= \tilde{H}^{g}_{\partial M} - \tilde{H}^{g_{*}}_{\partial M} +   2 \Tr(C)  -\frac{2}{\lambda_{*}} \sqrt{a^{44}} \Tr(S)  +O(t).
\end{align}
\begin{proof}
Since $\tilde{H}^{g}_{\partial M} = 2 \Tr(C)$, we can apply Proposition \ref{mccomparisong} and arguments in Corollary 5.2 in \cite{koyao} and obtain the formula (\ref{mcgeuclcomparison}).
\end{proof}
\end{cor}
We define the limit of the adjusted mean curvature on the foliation by $\tilde{H}_{p_{+}} : = \lim_{t \rightarrow 0} \tilde{H}_{\Sigma_{t}}^{g}(x)$.  We obtain the following as a limit of the adjusted mean curvature on leaves near $p_{+}$.
\begin{cor} \label{limitmcleavecor}
The formula of the limit of the normalized adjusted mean curvature on the foliation  $\hat{H}_{p_{+}}$ holds as
\begin{align} \label{limitmcleave}
    \tilde{H}_{p_{+}} &= \tilde{H}^{g}_{\partial M}(p_{+}) - \tilde{H}^{g_{*}}_{\partial M}(p_{+}) +   2 \Tr(C) -  -\frac{2}{\lambda_{*}} \sqrt{a^{44}}\Tr(S).
\end{align}
\begin{proof} We directly obtain (\ref{limitmcleave}) by sending $t$ to $0$ in (\ref{mcgeuclcomparison}). 
\end{proof}
\end{cor}
We now characterize the conditions of $\tilde{H}_{p_{+}}>0$, which gives a foliation with a strictly positive mean curvature near $p_{+}$. 
\begin{prop} \label{posmc} $\tilde{H}_{p_{+}}\ge0$ and the equality holds if and only if $g(p_{+})|_{\partial M}$ is a multiple of an identity matrix and $\tilde{H}_g^2g|_{\partial M}=\tilde{H}_0^2g_{*}|_{\partial \bar{M}}$ at $p_{+}$. 
\begin{proof}
Note that $\tilde{H}_g^2g|_{\partial M} = \lambda_{*}^{-2} \tilde{H}_g^2g|_{\partial M} \geq\lambda_{*}^{-2} \hat{H}_0^2g_{*}|_{\partial \bar{M}} = \tilde{H}_0^2g_{*}|_{\partial \bar{M}}$, we have
\begin{equation*}
    \tilde{H} \ge \max \bigg \{ \sqrt{\frac{(g_{*})_{ii}}{a_{ii}}} \bigg \}_{i=1}^{3} \tilde{H}_{0}.
\end{equation*}
Without loss of generality, let us take $a_{11} = \min \{ a_{ii} \}$. By applying $\tilde{H}_0 = 2 \Tr(C)$ and Corollary \ref{limitmcleavecor}, we have
\begin{align} \nonumber\tilde{H}_{p_{+}} &= \tilde{H}^{g}_{\partial M}(p_{+}) - \tilde{H}^{g_{*}}_{\partial M}(p_{+}) +   2 \Tr(C)   -\frac{2}{\lambda_{*}} \sqrt{a^{44}}\Tr(S)  \\
\nonumber &\ge 2 \bigg( \frac{1}{\lambda_{*}\sqrt{a_{11}}} -1 \bigg) \Tr(C) +   2 \Tr(C) -  \frac{2}{\lambda_{*}} \sqrt{a^{44}} \Tr(S)  \\ &= \nonumber \frac{2}{\lambda_{*}} \bigg( \frac{1}{\sqrt{a_{11}}} \Tr(C) - \sqrt{a^{44}} \Tr(S)  \bigg) \\&= \label{tracemc} \frac{2}{\lambda_{*}} \bigg( \frac{1}{\sqrt{a_{11}}} \Tr(C) - \sqrt{a^{44}} \Tr\left( \left(D^{\frac{1}{2}}C^TCD^{\frac{1}{2}}\right)^{\frac{1}{2}} \right)  \bigg).
\end{align}
Since $\left(D^{\frac{1}{2}}C^TCD^{\frac{1}{2}}\right)^{\frac{1}{2}} \le 1/ (a_{11} a^{44}) C^TC$ we have
\begin{align} \nonumber
    \left(D^{\frac{1}{2}}C^TCD^{\frac{1}{2}}\right)^{\frac{1}{2}} &\le \frac{1}{\sqrt{a_{11} a^{44}}} \left( C^TC \right)^\frac{1}{2} \\ \label{symmetrymat} &\le \frac{1}{\sqrt{a_{11} a^{44}}} C,
\end{align}
where (\ref{symmetrymat}) follows from the fact that $C$ is a symmetric matrix. By taking the trace to (\ref{symmetrymat}) and applying to 
(\ref{tracemc}), we obtain $\tilde{H}_{p_{+}} \ge 0$ and the equality holds if and only if $g(p_{1})|_{\partial M}$ is a multiple of an identity matrix and $\tilde{H}_g^2g|_{\partial M}=\tilde{H}_0^2g_{*}|_{\partial \bar{M}}$ at $p_{+}$.
\end{proof}
\end{prop}
We prove the existence of local foliation near $p_{+}$ whose mean curvature on leaves is nonnegative and the contact angle is larger than the angle function induced by Euclidean slices.
\begin{prop} \label{nonzeromcfoliation}
Assume $g(p_{1})|_{\partial M}$ is not a multiple of an identity matrix. Then there exists a foliation $\Sigma_{s,t}$ near $p_{+}$ such that $\tilde{H}_{s,t}>0$ and $\theta_{s,t}^{g}(x) > \overline{\theta}(x)$ for each $x \in \Sigma_{s,t} \cap \partial M$.
\begin{proof} By Proposition \ref{posmc}, we have $\tilde{H}_{p_{+}}>0$. Hence, for small $0<s<s_{0}$ and $0<t<t_{0}$, we have a foliation $\Sigma_{s,t}$ which satisfies $\tilde{H}_{s,t}>0$ and $\theta_{s,t}^{g}(x) > \overline{\theta}(x)$ on boundary by Proposition \ref{firstcontactanglecomp}.
\end{proof}
\end{prop}
We obtain a nonnegative adjusted mean curvature foliation when $g(p_{1})|_{\partial M}$ is a multiple of an identity matrix, and we can apply the same argument to the adjusted mean curvature $\hat{H}$ without normalization with Proposition 5.28 in \cite{CW1}.
\begin{prop} \label{nontrivialfirstminimizer} Assume $g(p_{1})|_{\partial M}$ is a multiple of an identity matrix. Then there exists a nonnegative constant adjusted mean curvature foliation $\Sigma_{t}$ with prescribed angle $\overline{\theta}_{1}(x)$ along $\partial \Sigma_{t}$. Moreover, there exists a nontrivial minimizer $\Omega \neq \emptyset$ of $\mathcal{A}_{\theta_{1}}$ in $M$.
\begin{proof} The first statement directly follows from arguments in Proposition 5.28 in \cite{CW1} and Proposition 5.5 in Ko-Yao \cite{koyao}. Since $\frac{d}{dt}|_{t=0}\mathcal A_{\bar{\theta}_1}(\Omega_{t})<0$ for nonnegative constant adjusted mean curvature foliation by the definition of $\mathcal A_{\bar{\theta}_1}$, we can apply arguments Proposition 5.6 in Ko-Yao \cite{koyao} with Proposition \ref{nonzeromcfoliation}, and we obtain the existence of a nontrivial minimizer $\Omega \neq \emptyset$ of $\mathcal{A}_{\theta_{1}}$.
\end{proof}
\end{prop}
\subsection{Second Slicing}
Let $\Sigma_1$ be a connected component of a non-trivial minimizer for the first capillary functional $\mathcal{A}_{\bar{\theta}_{1}}$ by the previous slicing. We denote the metric $g|_{\Sigma_{1}}$ by $h$. The second capillary functional is a weighted capillary functional, hence it is equivalent to considering a minimizer of capillary functional with respect to the metric $(\Sigma_{1}^{3}, f_{1} h)$, where $f_{1}$ is the first eigenfunction of the stability operator on $\Sigma_{1}$ from the first slicing. We denote $\tilde{h}=f_{1}h$.

Let us consider $p_{2} \in \partial \Sigma_{1}$ such that the $x_{4}$-coordinate of $p_{2}$ is maximum over $\partial \Sigma_{1}$. We have that $\nabla^{\partial \Sigma_{1}}x_{4}(p_{2})=0$ and hence $T_{p_{2}}\partial\Sigma_{1}$ agrees to $T_{p_{2}}\partial\bar{\Sigma}_{1}$, where $\bar{\Sigma}_{1}$ is a corresponding codimension $1$ slice in $\bar{M}$. After an appropriate rotation of the first $3$ axes, without loss of generality, we take the $x_{3}$-axis to achieve the maximum value of $x_{3}$-coordinate at $p_{2}$.

We denote the second fundamental form of $(\bar{\Sigma},\bar{h})$ in Euclidean coordinate at the point $p_2$ as $-2M$, where $M=(m_{ij})$ and $\bar{h}=g_{E}|_{\bar{\Sigma}}$ is again a Euclidean metric by definition. We can write $\partial\bar{\Sigma}$ as a graph near $p_2$ as:
\begin{align}
   \partial \bar{\Sigma}=\{(x_1,x_2,m_{ij}x_ix_j+O(|x|^3))\}.
\end{align}

Then there exists a local diffeomorphism $\Psi$ between a small neighborhood of $p_2$ in $\bar{\Sigma}$ and a small neighborhood of $p_2$ in $\Sigma$ satisfying the following:
    \begin{enumerate}
        \item $\Psi(p_2)=p_2$ and $D\Psi|_{p_2,T_{p_2}\partial \bar{\Sigma}}=id$;
        \item There exists a neighborhood $N_{p_{2}}$ of $p_{2}$ such that $\Psi ( \bar{\Sigma} \cap N_{p_{2}})\subset \Sigma$;
        \item The pullback metric of $\tilde{h}$ at $p_2$ in the Euclidean coordinate can be represented by
        \begin{align}
        \tilde{h}_0=\left(\begin{matrix}
            \tilde{h}_{11} & 0 & \tilde{h}_{13}\\
            0 & \tilde{h}_{22} & \tilde{h}_{23}\\
            \tilde{h}_{31} & \tilde{h}_{32} & \tilde{h}_{33}
        \end{matrix}\right).
        \end{align}
    \end{enumerate}

Define 
    \begin{align*}
        Q=\left(\begin{matrix}
            \frac{1}{\tilde{h}_{11}\tilde{h}^{33}} &0 \\
            0 & \frac{1}{\tilde{h}_{22}\tilde{h}^{33}}
        \end{matrix}\right).
    \end{align*}
     Since $Q$ is a positive definite matrix, we define 
    \[
    R=\left(Q^{1/2}M^{T}MQ^{1/2}\right)^{1/2}
    \]
    and
    \[
    U=Q^{-1/2}RQ^{-1/2}.
    \]
   Finally, we take
    \[
   n_{ij}=\frac{U_{ij}}{m_{ij}},
    \]
    when $m_{ij}=0$, we take $n_{ij}=0$.

We can also assume that $\tilde{h}=\tilde{h}_0+O(t)$, and we construct the following graphical foliation near $p_2$ as
\begin{align}
    \Gamma_{s,t}:=\{\varphi_{s,t}(x)=(tx_1+O(t^3),tx_2+O(t^3),t^2\left\{m_{ij}(n_{ij}(1+s)-1)x_ix_j\right\}+O(t^4)): x\in F_s\},
\end{align}
\[
F_s:=\{x\in\mathbb R^2: m_{ij}n_{ij}x_ix_j<(1+s)^{-1}\}.
\]

We use $\cos\rho_{s,t}(\Psi(x))$ to denote the contact angle of $\Gamma_{s,t}$ at $x$ with respect to the metric $\tilde{h}$, $\cos\bar{\rho}(x)$ to denote the prescribed angle function $\cos\bar{\theta}_2$ at the point $\varphi_{s,t}(x)\in \Gamma_{s,t}$ in $\partial\bar{\Sigma}$, and $\cos\bar{\rho}(\Psi(x))$ to denote the prescribed angle function $\cos\bar{\theta}_2$ at the point $\Psi(\varphi_{s,t}(x))\in\Gamma_{s,t}$ in $\partial\Sigma$. 

\begin{prop}
    We have the following comparison of contact angles
    \begin{align}
        \cos\rho_{s,t}(\Psi(x))=\cos\bar{\rho}(x)-4st^2\sum_{i}\frac{(m_{ij}n_{ij}x_j)^2}{\tilde{h}_{ii}\tilde{h}^{33}}+O(s^2)+O(t^3)
    \end{align}
\end{prop}
\begin{proof}
    By straightforward computations as in dimension $3$ \cite{koyao}, we have
    \[
    \cos\rho_{s,t}(\Psi(x))=1-2(1+2s)t^2\sum_{i}\left(\frac{(\sum_jm_{ij}n_{ij}x_j)^2}{\tilde{h}_{ii}\tilde{h}^{33}}\right)+O(s^2)+O(t^3),
    \]
    and
    \[
    \cos\bar{\rho}(x)=1-2t^2\sum_{i}(m_{ij}x_j)^2+O(t^3).
    \]
    By definition of $N$, we have the following identities:
    \begin{align*}
        \sum_i\frac{m_{ik}m_{il}n_{ik}n_{il}}{\tilde{h}_{ii}\tilde{h}^{33}}&=\sum_i\frac{U_{ik}U_{il}}{\tilde{h}_{ii}\tilde{h}^{33}}\\
        &=(UQU)_{kl}\\
        &=(M^TM)_{kl}\\
        &=\sum_{i}m_{ik}m_{il},
    \end{align*}
    for any $k,l\in\{1,2\}$. The desired estimates then follow.
\end{proof}

\begin{lem}
    We have the following estimate:
    \begin{align}
        \cos\bar{\rho}(x)-\cos\bar{\rho}(\Psi(x))=O(t^3).
    \end{align}
\end{lem}
\begin{proof}
    By Taylor expansion, we have
    \begin{align*}
        \cos\bar{\rho}(\Psi(x))&-\cos\bar{\rho}(x)\\
        &=-\sin\bar{\theta}_2(\varphi_{s,t}(x))|\Psi(\varphi_{s,t}(x))-\varphi_{s,t}(x)|+O(|\Psi(\varphi_{s,t}(x))-\varphi_{s,t}(x)|^2).
        \end{align*}
        Note that
        \begin{align*}
            |\Psi(\varphi_{s,t}(x))-\varphi_{s,t}(x)|&=|0+D\Psi|_{p_2}\cdot \varphi_{s,t}(x)-\varphi_{s,t}(x)+O(t^2)\\
            &=O(t^2),
        \end{align*}
        we used the fact that $D\Psi|_{p_2}=Id$ and we denote $p_2=0$ for simplicity.
        
        Also, we have the Taylor expansion
        \begin{align*}
            \sin\bar{\theta}_2(\varphi_{s,t}(x))&=0+\cos\bar{\theta}_2(p_2)|\varphi_{s,t}(x)|+O(|\varphi_{s,t}(x)|^{2})\\
            &=O(t).
        \end{align*}

        Therefore, we have 
        \[
        \cos\bar{\rho}(\Psi(x))-\cos\bar{\rho}(x)=O(t^3).
        \]
\end{proof}

\begin{cor}
     We have the following comparison of contact angles
    \begin{align}
        \cos\rho_{s,t}(\Psi(x))=\cos\bar{\rho}(\Psi(x))-4st^2\sum_{i}\frac{(m_{ij}n_{ij}x_j)^2}{\tilde{h}_{ii}\tilde{h}^{33}}+O(s^2)+O(t^3).
    \end{align}
\end{cor}

Using almost identical computations as in Ko-Yao \cite{koyao}, we have the following lemma.

\begin{lem}
    Suppose $\Gamma_t$ is a family of hypersurfaces which is parametrized as 
    \[
    \Gamma_t:=\{(tx_1,tx_2,t^2\{m_{ij}x_ix_j+O(t^2)\}): x\in E_s\}
    \]
    then the mean curvature under the metric $\tilde{h}$ is 
    \[
    H_t(x)=-\sum_i\frac{2m_{ii}}{\tilde{h}_{ii}\sqrt{\tilde{h}^{33}}}+O(t).
    \]
\end{lem}

As a corollary, we compute the mean curvature of $\Gamma_{s,t}$.
\begin{cor}
    For small $s,t>0$, the mean curvature of $\Sigma_{s,t}$ is computed as
    \begin{align}
       H_{s,t}^{\tilde{h}}(x)=-\frac{1}{\sqrt{\tilde{h}^{33}}}\left(\frac{2m_{11}(n_{11}-1)}{\tilde{h}_{11}}+\frac{2m_{22}(n_{22}-1)}{\tilde{h}_{22}}\right)+O(s)+O(t).
    \end{align}
    Let $s,t\to 0$, we have 
    \begin{align}
H_{p_2}:=\lim_{(s,t)\to(0,0)}H_{s,t}^{\tilde{h}}(x)=-\frac{1}{\sqrt{\tilde{h}^{33}}}\left(\frac{2m_{11}(n_{11}-1)}{\tilde{h}_{11}}+\frac{2m_{22}(n_{22}-1)}{\tilde{h}_{22}}\right).    \end{align}
\end{cor}

We have following formula relating a mean curvature on $\partial \Sigma$ and that on $\partial M$. It follows from analogous argument with Lemma 2.1 in \cite{CW1} and the definition of the first eigenfunction of a minimizer.
\begin{prop} Suppose $\Sigma$ is a minimizer of $\mathcal{A}_{\bar{\theta}_{1}}$ on $M$ and $f_{1}$ is the first eigenfunction of its stability operator. Then we have a following formula at $p_{2} \in \partial \Sigma$.
\begin{align} \label{bdcurvature1steigen}
H_{\partial \Sigma}^{\tilde{h}}(p_{2}) = \frac{f_{1}^{-\frac{1}{2}}}{\sin \bar{\theta}_{1}} \left( \hat{H}_{\partial M}(p_{2}) - \partial_{\nu} \bar{\theta}_{1} \right).
\end{align}
\begin{proof}
    We compute as
    \begin{align}
    H_{\partial \Sigma}^{\tilde{h}}(p_{2}) &= f_{1}^{-\frac{1}{2}}(H_{\partial \Sigma}^{h}(p_{2}) + \langle \nabla \log f_{1}, \, \eta\rangle) \label{confdeformmc} \\ \nonumber &= f_{1}^{-\frac{1}{2}}\left(\sum_{i=1}^{2} \langle \nabla_{\tau_{i}} \tau_{i}, \, \eta \rangle + \langle \nabla \log f_{1}, \, \eta\rangle\right) \\ \label{nvectordecomp}
    &= f_{1}^{-\frac{1}{2}}\left(\frac{1}{\sin \bar{\theta}_{1}}\sum_{i=1}^{2} \langle \nabla_{\tau_{i}} \tau_{i}, \,   X - \cos \bar{\theta}_{1} N_{1}  \rangle + \langle \nabla \log f_{1}, \, \eta\rangle\right) \\ \label{minimalitysigma}
    &= f_{1}^{-\frac{1}{2}}\left(\frac{1}{\sin \bar{\theta}_{1}} (H_{\partial M}(p_{2}) - A_{\partial M}(\nu,\nu) + \cos \bar{\theta}_{1} H_{\Sigma_{1}}) + \cot \bar{\theta}_{1} A_{\Sigma}(\eta,\eta)+ \langle \nabla \log f_{1}, \, \eta\rangle\right)  \\ \label{firsteigenfunctionsigma} &= \frac{f_{1}^{-\frac{1}{2}}}{\sin \bar{\theta}_{1}} \left( \hat{H}_{\partial M}(p_{2}) - \partial_{\nu} \bar{\theta}_{1} \right),
    \end{align}
    where (\ref{confdeformmc}) follows from the mean curvature by conformal deformation, (\ref{nvectordecomp}) is by $X = \sin \bar{\theta}_{1} \eta + \cos \bar{\theta}_{1} N_{1}$ at $p_{2}$, (\ref{minimalitysigma}) is by the minimizing property of $\mathcal{A}_{\bar{\theta}_{1}}$, and (\ref{firsteigenfunctionsigma}) is by the definition of $f_{1}$ with (\ref{secondvariationtheta1}) and (\ref{q1def}).  
\end{proof}
\end{prop}
We obtain the following curvature comparison that holds at an extreme point $p_{2}$.
\begin{prop} [Curvature comparison at an extreme point] \label{p2curvcomparison}
    At the point $p_2$, we have the following accurate curvature comparison 
    \begin{align}\label{eq: mean curvature comparison}
        \hat{H}_{\partial M}&-\partial_{\nu}\bar{\theta}_1\geq \\
        \nonumber& \frac{\sqrt{(\sqrt{\tilde{h}_{11}}\RN{2}_{Eucl}(\bar{\tau}_2,\bar{\tau}_2)+\sqrt{\tilde{h}_{22}}\RN{2}_{Eucl}(\bar{\tau}_1,\bar{\tau}_1))^2+(\sqrt{\tilde{h}_{11}}-\sqrt{\tilde{h}_{22}})^2\RN{2}_{Eucl}(\bar{\tau}_1,\bar{\tau}_2)^2}}{\sqrt{\tilde{h}_{11}\tilde{h}_{22}}}.
    \end{align}
    The equality holds if $\tilde{h}_{11} = \tilde{h}_{22}$.
\end{prop}
\begin{proof}
    Without loss of generality, we assume $\tilde{h}_{11}\geq \tilde{h}_{22}$. Let $k_1=\sqrt{\frac{\tilde{h}_{22}}{\tilde{h}_{11}}}=\sqrt{\frac{h_{22}}{h_{11}}}$, we have
    \begin{align}
        \nonumber h_{11}h_{22}(\hat{H}_{\partial M}-\partial_{\nu}\bar{\theta}_1)^2&= h_{11}h_{22}(\hat{H}_{\partial M}-\RN{2}(\nu, \bar{\nu}))^2\\ \nonumber & \ge h_{11} h_{22} (\sup_{\text{span} \{\tau_{1}, \tau_{2}\}= \text{span} \{\bar{\tau}_{1}, \bar{\tau}_{2}\}} \RN{2}_{Eucl}(\tau_{1},\bar{\tau}_{1}) +\RN{2}_{Eucl}(\tau_{2},\bar{\tau}_{2}))^{2} \\ \label{tracesupremum} &=(\sqrt{h_{22}}\RN{2}_{11}+\sqrt{h_{11}}\RN{2}_{22})^2+(\sqrt{h_{11}}-\sqrt{h_{22}})^2\RN{2}_{12}^2,
    \end{align}
    where we use $\RN{2}_{ij}=\RN{2}_{Eucl}(\bar{\tau}_i,\bar{\tau}_j)$ for simplicity. The equality (\ref{tracesupremum}) comes from Von Neumann inequality of the trace. The equality condition follows from the equality condition $\hat{H}_{\partial M}^{2} g|_{F^{\#}\RN{2}_{+}} = \hat{H}_{\partial M}^{2} g_{Eucl}|_{\RN{2}_{+}}$ of Proposition \ref{prop:local estimates}.
\end{proof}
With the curvature comparison Proposition \ref{p2curvcomparison}, we have the following nonnegative mean curvature foliation.
\begin{cor} \label{cor: non-negative asymptotics}
    \begin{align}\label{eq: non-negative asymptotics}
H_{p_2}:=\lim_{(s,t)\to(0,0)}H_{s,t}^{\tilde{h}}(x)=-\frac{2}{\sqrt{\tilde{h}^{33}}}\left(\frac{m_{11}(n_{11}-1)}{\tilde{h}_{11}}+\frac{m_{22}(n_{22}-1)}{\tilde{h}_{22}}\right)\geq 0.
    \end{align}
    The equality holds if $\tilde{h}_{11} = \tilde{h}_{22}$.
\end{cor}
\begin{proof}
    By definition of $M$, we know
    \[
    m_{ij}=-\frac{1}{2\sin\bar{\theta}_1(p_2)}\RN{2}_{Eucl}(\bar{\tau}_i,\bar{\tau}_j).
    \]
    By definition of $n_{ij}$, we have
    \[
    n_{ii}=\sqrt{\tilde{h}^{33}}\frac{\tilde{h}_{ii}(m_{11}m_{22}-m_{12}^2)+\sqrt{\tilde{h}_{11}\tilde{h}_{22}}(m_{ii}^2+m_{12}^2)}{m_{ii}\sqrt{(\sqrt{\tilde{h}_{11}}m_{22}+\sqrt{\tilde{h}_{22}}m_{11})^2+(\sqrt{\tilde{h}_{11}}-\sqrt{\tilde{h}_{22}})^2m_{12}^2}}.
    \]

    Then we have
    \begin{align*}
        \sum_{i=1}^2&\frac{m_{ii}n_{ii}}{\sqrt{\tilde{h}^{33}}\tilde{h}_{ii}}=\\
        &\frac{2(m_{11}m_{22}-m_{12}^2)+\left(\sqrt{\tilde{h}_{11}/\tilde{h}_{22}}+\sqrt{\tilde{h}_{22}/\tilde{h}_{11}}\right)m_{12}^2+\sqrt{\tilde{h}_{22}/\tilde{h}_{11}}m_{11}^2+\sqrt{\tilde{h}_{11}/\tilde{h}_{22}}m_{22}^2}{\sqrt{(\sqrt{\tilde{h}_{11}}m_{22}+\sqrt{\tilde{h}_{22}}m_{11})^2+(\sqrt{\tilde{h}_{11}}-\sqrt{\tilde{h}_{22}})^2m_{12}^2}}\\
        &=\frac{\sqrt{(\sqrt{\tilde{h}_{11}}m_{22}+\sqrt{\tilde{h}_{22}}m_{11})^2+(\sqrt{\tilde{h}_{11}}-\sqrt{\tilde{h}_{22}})^2m_{12}^2}}{\sqrt{\tilde{h}_{11}\tilde{h}_{22}}}\\
        &=\frac{\sqrt{(\sqrt{{h}_{11}}m_{22}+\sqrt{{h}_{22}}m_{11})^2+(\sqrt{{h}_{11}}-\sqrt{{h}_{22}})^2m_{12}^2}}{\sqrt{{h}_{11}{h}_{22}}}
    \end{align*}

Also note that
\begin{align*}
    \sum_{i=1}^2&\frac{m_{ii}}{\sqrt{\tilde{h}^{33}}\tilde{h}_{ii}}=\frac{1}{2}H^{\tilde{h}}_{\partial\Sigma}=\frac{f_1^{-1/2}(p_2)}{2\sin\bar{\theta}_1(p_2)}\left(\hat{H}_{\partial M}-\partial_{\nu}\bar{\theta}_1\right)\\
    &=\frac{1}{2\sin\bar{\theta}_1}\left(\hat{H}_{\partial M}-\partial_{\nu}\bar{\theta}_1\right),
\end{align*}
where we rescale $f_1$ to be one at the point $p_2$. We now conclude that \eqref{eq: non-negative asymptotics} is equivalent to \eqref{eq: mean curvature comparison} and the equality follows from the same condition.
\end{proof}
Since we have the equality condition $\tilde{h}_{11} = \tilde{h}_{22}$ at $p_{2}$, we can apply the construction of CMC foliation which does not depend on the dimension as in \cite{koyao}.
\begin{cor} \label{nontrivialsecondminimizer}
    There exists a non-negative mean curvature foliation $\Gamma_t$ with prescribed angle $\bar{\theta}_2$ along $\partial\Gamma_t$, and there exists a nontrivial minimizer of $\mathcal A_{\bar{\theta}_2}$.
\end{cor}
\begin{proof}
    We obtain the nontrivial barrier when $\tilde{h}_{11} \neq \tilde{h}_{22}$ from Corollary \ref{cor: non-negative asymptotics} and can apply the arguments in Proposition 5.6 in \cite{koyao}. The CMC foliation arguments are identical to Proposition 5.5 in \cite{koyao}.
\end{proof}
\begin{proof}[Proof of Theorem \ref{main slicing thm} and hyperbolic] In both Euclidean and Hyperbolic cases, by Proposition  \ref{nontrivialfirstminimizer} and Corollary \ref{nontrivialsecondminimizer}, there exists a nontrivial minimizer for each slicing. Moreover, the conditions (1) and (2) are satisfied by our choice of the prescribed angle functions $\{ \bar{\theta}_{i}\}_{i=1,2}$ and the regularity of the minimizers in Chodosh-Edelen-Li \cite{chodosh2024improved}.
\end{proof}

\section{Proof of Main Theorem}
In this section, we prove the local splitting theorem based on the existence of nontrivial minimizer of capillary functionals $\mathcal{A}_{\bar{\theta}_{1}}$ and $\mathcal{A}_{\bar{\theta}_{2}}$, and complete to prove the main rigidity statement. Assuming the nontrivial minimizer is flat, we prove that there is a dense set of slices which are all flat Euclidean slices on both sides of the minimizer. In the hyperbolic comparison case, we prove that there is a dense set of horospherical slices on both sides, with an assumption of the existence of a horospherical minimizer. We apply this technique which first appeared in Carlotto-Chodosh-Eichmair \cite{carlotto2016effective}, whose conformal deformation of metrics studied earlier in Ehrlich \cite{ehrlich1976metric} and Liu \cite{liu20133}. The compactness argument is based on the recent regularity theory of minimizer of capillary functionals by Chodosh-Edelen-Li \cite{chodosh2024improved} and \cite{chodosh2025weiss}.

\subsection{Curvature estimate of minimizer of capillary functionals} We prove curvature estimates of  a minimizer of capillary functionals when the dimension of an ambient manifold is bounded above by $4$. The argument is based on the standard rescaling argument. Recall that the rescaling limit of a minimizer of capillary functionals is capillary minimal hypersurface with a constant contact angle in a half-space. We first prove the following Bernstein-type lemma when the capillary angle is constant based on the classification of minimizing cone of a capillary area functional by Chodosh-Edelen-Li (Theorem 1.2 of \cite{chodosh2024improved}).
\begin{lem} \label{capillarybernstein}
    Suppose $n+1 \le 4$ and $\theta \in (0,\pi)$. $\Sigma^{n}$ is a properly embedded capillary minimal hypersurface with a contact angle $\theta$ in $\mathbb{R}^{n+1}_{+}$ which minimizes $\mathcal{A}^{\theta}$-functional. Then $\Sigma^{n}$ is a flat half-hyperplane.  
    \begin{proof} Suppose $\Sigma^{n}$ is not a half-hyperplane and we prove by contradiction. Take $x \in \Sigma^{n} \cap \partial \mathbb{R}^{n+1}$ and consider a blow-down limit $\Sigma_{x,\infty} = \lim_{r \rightarrow \infty} \mu_{r,x} (\Sigma)$, where $\mu_{r,x}(y) = (y-x)/r$ is a scaling map by scale $1/r$ for $y \in \mathbb{R}^{n+1}_{+}$. Then by the classification of area-minimizing cone of capillary area functional by Chodosh-Edelen-Li (Theorem 1.2 of \cite{chodosh2024improved}), we have $\Sigma_{x,\infty}$ is a half-hyperplane. Now let us consider a density functional
    \begin{equation*}
        \Theta_{\theta}(x,r)(\Sigma) = \frac{|\Sigma \cap B_{r}(x)| - \cos \theta \,|\partial \mathbb{R}^{n+1}_{+} \cap \partial \Omega_{x,\infty} \cap B_{r}(x)|}{ \omega_{n-1}r^{n}/n},
    \end{equation*}
    where $\Omega_{x,\infty}$ is a fixed domain separated by $\Sigma_{x,\infty}$ in $\mathbb{R}^{n+1}_{+}$. Then by the monotonicity formula of the capillary functional by Kagaya-Tonegawa \cite{kagaya2017fixed}, $\Theta_{\theta}(x,r)$ is a monotonically increasing functional by $r$ for a fixed $x$. Since $\Theta_{\theta}(x,r)(\Sigma)=  \Theta_{\theta}(x,1)(\mu_{r,x} (\Sigma))$, the density function $d_{x}(r) :=  \Theta_{\theta}(x,1)(\mu_{r,x} (\Sigma))$ is a monotonically increasing function by $r$. Also, since the area-minimizing capillary minimal hypersurfaces have a Euclidean volume growth, there exists a density at infinity $d_{x,\infty} := \lim_{r \rightarrow \infty} d_{x}(r) < \infty$. Since $\Sigma_{x,\infty}$ is a half-hyperplane, $d_{x,\infty} = (1- \cos \theta)/2$. Also, since $d_{x,0} := \lim_{r \rightarrow 0} d_{x}(r) = (1- \cos \theta)/2$, we have $d_{x}(r) \equiv (1- \cos \theta)/2$ and this gives that $\Sigma$ is a flat half-hyperplane and this contradicts our assumption.
    \end{proof}
\end{lem}
We now prove the curvature estimate based on Lemma \ref{capillarybernstein}.
\begin{thm} \label{thm : curv estimate} Suppose $n+1 \le 4$ and $\Sigma$ to be a minimizing capillary
hypersurface in $(M,g)$ with prescribed contact angle $\bar\theta$ and prescribed mean curvature $\mu$ which minimizes a functional $\mathcal{A}_{\overline{\theta}}$ on $(M,g)$ whose boundary is mean convex. Also, assume that there exists $\epsilon>0$ such that $\overline{\theta}(x) \in (\epsilon, \pi - \epsilon)$ at point $x \in \Sigma$. Then
\begin{equation*}
    \sup_{x \in \Sigma} |A_{\Sigma}|(x) \le C(M,g),
\end{equation*}
where $C$ can be chosen only depending on $M$ and the metric $g$ and which is uniform by a compact family of smooth metrics $g$.
\begin{proof} Let us prove by contradiction and assume that there is a sequence of $\Sigma_{j}$ such that $\lambda_{j} :=  \sup |A_{\Sigma_{j}}| \rightarrow \infty$ as $j$ goes to infinity. Then there exists a sequence of points $\{ x_{j} \}$ such that $|A_{\Sigma_{j}}|(x_{j}) = \lambda_{j}$. Consider the rescaled sequence of capillary hypersurfaces $\mu_{1/\lambda_{j},x_{j}}(\Sigma_{j}) = \lambda_{j} (\Sigma_{j}- x_{j})$ with bounded prescribed mean curvature. Then we have $|A_{\mu_{1/\lambda_{j},x_{j}}(\Sigma_{j})}|(0) = 1$ for any $j$. Since $|A_{\mu_{1/\lambda_{j},x_{j}}(\Sigma_{j})}| \le 1$ everywhere, the manifold $\mu_{1/\lambda_{j},x_{j}}(M)$ smoothly and graphically converges to a limit capillary hypersurface $\Sigma_{\infty} \subseteq \mathbb{R}^{n+1}_{+}$ by elliptic PDE theory, and we have $A_{\Sigma_{\infty}} (0)=1$. Note that the contact angle $\overline{\theta}_{\infty}$ is a constant and satisfies $\overline{\theta}_{\infty} \in [\epsilon, \pi- \epsilon]$. We can apply Lemma \ref{capillarybernstein} and obtain that $\Sigma_{\infty}$ is a flat half-hypersurface and this contradicts to $A_{\Sigma_{\infty}} (0)=1$. Hence we have $\sup_{x \in \Sigma} |A_{\Sigma}|(x) \le C(M,g)$.
\end{proof}
\end{thm}
\begin{rem}
    We took a uniform angle bound $(\epsilon, \pi-\epsilon)$ of a contact angle $\overline{\theta}$ in Theorem \ref{thm : curv estimate} because we cannot guaranty Lemma \ref{capillarybernstein} when the contact angle degenerates on $\Sigma$. By our definition of contact angle functions $\bar{\theta}_{1}$ and $\bar{\theta}_{2}$, the contact angle degeneracy only occurs at the extreme points and this does not happen for the minimizers by maximum principle. Hence, Theorem \ref{thm : curv estimate} can be applied to all infinitesimally rigid solids we will construct later.
\end{rem}

\subsection{Infinitesimally rigid surfaces} We first analyze the equality cases from the local estimates in Section \ref{section: local estimates}. We generalize the rigidity analysis in dimension $3$ cases in Ko-Yao \cite{koyao}. We discuss the rigidity analysis of the second slice in Euclidean and hyperbolic at once, and separate the rigidity analysis of the first slice into two cases.

We start with the rigidity analysis of $\Sigma_2=\Gamma^{2}$. We call $\Gamma$ infinitesimally rigid if one component of $\Gamma$ which has a nonempty boundary satisfies the following:
 \begin{align}
       \chi(\Gamma)=1,\quad \nabla^{\Sigma}f_1=0, \quad \|A_{\Gamma}\|=0,\quad K_{\Gamma}=0 \quad \text{ on }\Gamma;\\
       \kappa_{\partial\Gamma}=\frac{\hat{H}_{\partial M}}{\hat{H}_{\partial \bar{M}}}\kappa_{\partial\bar{\Gamma}},\quad \quad g(\eta,N_2)=\cos\bar{\theta}_2\quad \text{ on } \partial\Gamma,
   \end{align}
   and
   \[
   \tau_1\parallel\bar{\tau}_1,\quad \tau_2\parallel\bar{\tau}_2, \quad \nu\parallel\bar{\nu}.
   \]
Also, we denote $\tau_{3} = \nu$ and $\overline{\tau}_{3} = \overline{\nu}$ for convenience of notation.
\begin{thm}[Rigidity Analysis of $\Gamma$]\label{thm: rigidity of Gamma}
   Suppose $\Gamma$ is a non-trivial minimizer of $\mathcal A_{\bar{\theta}_2}$ in $\Sigma$, and $\Sigma$ is a nontrivial minimizer of $\mathcal A_{\bar{\theta}_1}$, then $\Gamma$ is an infinitesimally rigid surface in $\Sigma$. Also, we obtain \[
   A_{\Sigma}= \begin{cases}&0 \quad \text{for Euclidean comparison}\\ &-cI \quad \text{for Euclidean comparison}
   \end{cases}
   \] on $\Gamma$. Moreover, we have 
    \[
   \tau_i=\frac{\hat{H}_{\partial M}}{\hat{H}_{\partial \bar{M}}}\bar{\tau}_i.
   \]
   for $i=1,2$.
\begin{proof}
    By the local estimate (Corollary \ref{cor: key estmate}) for $n=2$, the winding number argument for $n=2$ (Proposition \ref{prop: winding number}) and the second variation formula (\ref{eq: bottom slice second variation}), we have
    \begin{align*}
    2\pi&\geq \int_{\Gamma}-\frac{3}{4}|\nabla^{\Gamma}\log f_1|^2-\frac{1}{2}\left(R_g+\|A_{\Sigma}\|^2+\|A_{\Gamma}\|^2+H_{\Gamma}^2\right)+K_{\Gamma}+\int_{\partial\Gamma}k_g\\
    &=\int_{\partial\Gamma}\frac{1}{\sin\bar{\theta}_1\sin\bar{\theta}_2}\left(H_{\partial M}-\partial_{\nu}\bar{\theta}_1-\sin\bar{\theta}_1\partial_{\tau_2}\bar{\theta}_2\right)\\
    &\geq 2\pi.
    \end{align*}
    By tracing equalities above, we have
    \[
    \chi(\Gamma)=1,\quad \nabla^{\Gamma}f_1=0, \quad \|A_{\Gamma}\|=0,\quad \|A_{\Sigma}\|=0,\quad R_g=0 \text{ and }H_{\Gamma}=0.
    \]
     Similarly, in the hyperbolic setting
    \begin{align*}
    2\pi & \geq \int_{\Gamma} -6 (c^{2}-\rho^{2}) + K_{\Gamma}+\int_{\partial\Gamma}k_g \\&\geq \int_{\Gamma}-\frac{3}{4}|\nabla^{\Gamma}\log f_1|^2-\frac{1}{2}\left(R_g+\|A_{\Sigma}\|^2+\|A_{\Gamma}\|^2+H_{\Gamma}^2\right)+K_{\Gamma}+\int_{\partial\Gamma}k_g\\
    &=\int_{\partial\Gamma}\frac{1}{\sin\bar{\theta}_1\sin\bar{\theta}_2}\left(\hat{H}_{\partial M}-\partial_{\nu}\bar{\theta}_1-\sin\bar{\theta}_1\partial_{\tau_2}\bar{\theta}_2\right)\\
    &\geq 2\pi.
    \end{align*}
    By tracing the inequalities above and (\ref{eqn:equalitytracinghyperbolic}), we have
    \[
    \chi(\Gamma)=1,\quad \nabla^{\Gamma}f_1=0, \quad \|A_{\Gamma}\|=0,\quad \rho \equiv c,\quad A_{\Sigma}=-cI,\quad R_g=-12c^{2} \text{ and }H_{\Gamma}=0.
    \]
    By $H_{\Gamma}=0$ and the first variation formula (Proposition \ref{prop: second slicing variation formula}), we have $\nabla^{\Sigma} f_{1}=0$. The boundary angle condition $g(\eta,N_2)=\cos\bar{\theta}_2$ follows from the first variation formula. Moreover, we have $K_{\Gamma}=0$ since a constant function is a Jacobi vector field of equation with Robin boundary condition in the interior of $\Gamma$.
    
    By the boundary part of the equation with Robin boundary condition, we have
    \begin{align}
        \nonumber \kappa_{\partial \Gamma} &= \frac{1}{\sin\bar{\theta}_1\sin\bar{\theta}_2}\left(\hat{H}_{\partial{M}}-\partial_{\nu}\bar{\theta}_1-\sin\bar{\theta}_1\partial_{\tau_2}\bar{\theta}_2\right) \\ \label{eq:boundaryrigidity}
        &= \frac{|\nu|_{g_{Eucl}}}{\sin\bar{\theta}_1\sin\bar{\theta}_2}\left(\hat{H}_{\partial \bar{M}}-\partial_{\bar{\nu}}\bar{\theta}_1-\sin\bar{\theta}_1\partial_{\bar{\tau}_2}\bar{\theta}_2\right) \\ \nonumber &= |\nu|_{g_{Eucl}} \kappa_{\partial \bar{\Gamma}} =  \frac{\hat{H}_{\partial M}}{\hat{H}_{\partial \bar{M}}}\kappa_{\partial\bar{\Gamma}},
    \end{align}where (\ref{eq:boundaryrigidity}) follows from the equality case of Corollary \ref{cor: key estmate}. The last statement follows by combining arguments in Corollary \ref{cor: key estmate} and the nonzero assumption on the second fundamental form.
    \end{proof}
    \end{thm}
Analogously, we proceed to the rigidity analysis of $\Sigma$. We divide into Euclidean and hyperbolic cases. In Euclidean setting, we now suppose that the induced metric $g|_{\Sigma}$ is a Euclidean metric as an induction hypothesis, which we will prove by global isometry statement later. We prove the infinitesimal rigidity of $\Sigma$ as follows.
\begin{prop} [Rigidity Analysis of $\Sigma$ : Euclidean case] \label{prop : rigidity sigma}
    $\Sigma$ is infinitesimally rigid in the sense that
    \begin{equation*}
        \text{Ric}_{M}(N_1,N_1)=0, \quad  \|A_{\Sigma}\|=0 \quad \text{on } \Sigma,
    \end{equation*}
    and $H_{\partial\Sigma}=\frac{H_{\partial M}}{\bar{H}_{\partial M}}H_{\partial\bar{\Sigma}}$ on $\partial \Sigma$.
    \begin{proof}
        We proved $\|A_{\Sigma}\|=0$ on $\Gamma$ in Theorem \ref{thm: rigidity of Gamma} and rigidity arguments by an induction hypothesis give that $\|A_{\Sigma}\|=0$ on $\Sigma$. Moreover, we assumed the rigidity, namely the flatness of $\Sigma$ which implies $R_{\Sigma}=0$, we have
        \begin{equation*}
        R_{\Sigma} = R_{M} - 2 \text{Ric}_{M}(N_1,N_1) - \|A_{\Sigma}\|^{2}+ \|H_{\Sigma}\|^{2}
        \end{equation*}
        and we obtain $\text{Ric}_{M}(N_1,N_1)=0$ from the minimality of $\Sigma$ in $M$. Since we have an isometry between $\Sigma$ and a Euclidean model, $H_{\partial\Sigma}=\frac{H_{\partial M}}{\bar{H}_{\partial M}}H_{\partial\bar{\Sigma}}$ on $\partial \Sigma$ follows.
    \end{proof}
\end{prop}
Similarly, we analyze the local rigidity of $\Sigma$ in the hyperbolic case. We assume that the induced metric $g|_{\Sigma}$ is isometric to the horospherical model as an induction hypothesis.
\begin{prop} [Rigidity Analysis of $\Sigma$ : Hyperbolic case] \label{prop : rigidity sigma hyperbolic}
    $\Sigma$ is infinitesimally rigid in the sense that
    \begin{equation*}
        \text{Ric}_{M}(N_1,N_1)=-3c^{2}, \quad  A_{\Sigma}=-cI \quad \text{on } \Sigma,
    \end{equation*}
    and $H_{\partial\Sigma}=\frac{\hat{H}_{\partial M}}{\hat{H}_{\partial \bar{M}}}H_{\partial\bar{\Sigma}}$ on $\partial \Sigma$.
      \begin{proof}
        We proved $A_{\Sigma}=-cI$ on $\Gamma$ in Theorem \ref{thm: rigidity of Gamma} and rigidity arguments by an induction hypothesis give that $A_{\Sigma}=-cI$ on $\Sigma$. Moreover, by our assumption we have the flatness of $\Sigma$ which implies $R_{\Sigma}=0$. Hence, again we have
        \begin{equation*}
        R_{\Sigma} = R_{M} - 2 \text{Ric}_{M}(N_1,N_1) - \|A_{\Sigma}\|^{2}+ \|H_{\Sigma}\|^{2}
        \end{equation*}
        and we obtain $\text{Ric}_{M}(N_1,N_1)=-3c^{2}$ from $A_{\Sigma}=-cI$ and $H_{\Sigma}= -3c$ in $M$. Since we have an isometry between $\Sigma$ and a Euclidean model, $H_{\partial\Sigma}=\frac{\hat{H}_{\partial M}}{\hat{H}_{\partial \bar{M}}}H_{\partial\bar{\Sigma}}$ on $\partial \Sigma$ follows.
    \end{proof}
    \end{prop}
\subsection{Global rigidity : Euclidean case} We now prove the main rigidity theorem in Euclidean comparison first. The idea is to construct an infinitesimally rigid capillary hypersurface in a conformally deformed metric of $g$ near the infinitesimal rigid hypersurface $\Sigma$. This gives a dense set of Euclidean capillary slices on both sides of $\Sigma$. We adapt the technique of \cite{carlotto2016effective} and \cite{liu20133} for this construction up to dimension $4$, where they developed the technique in dimension $3$ which can be extended directly (See \cite{li2024dihedral} for high dimensional free boundary setting). We take the construction of conformally deformed metric from Appendix \ref{appendixconf}.
\begin{proof}[Proof of Theorem \ref{main rigidity thm}] We prove the statement by induction, namely we prove the rigidity of $\Gamma$ in $\Sigma$ and those of $\Sigma$ in $M$ with the same idea, while the latter statement will be proven under the hypothesis of the former one. Since $\Gamma$ is a capillary minimal surface with a contact angle $\bar{\theta}_{2}$ in $\Sigma$ with the metric $g' = f_{1} g$, we prove at once by unifying the notation by $\Sigma^{n-1} \subseteq M^{n}$, the contact angle by $\bar{\theta}$, and the extreme point by $p_{+}$ without loss of generality, where $n=3,4$. Denote the region of $M$ separated by $\Sigma$ which contains $p_{+}$ by $M_{+}$. Denote the unit normal vector field on $\Sigma$ by $N$.

Now let us take $p \in M_{+}$ and take the metric $g(t)$ by applying Proposition \ref{prop : deformedmetric}. Then $(M_{+}, g(t))$ is a mean convex solid whose wedge part is on $\partial \Sigma$. Let us take $\mathcal{E}$ as a class of open sets with finite perimeter containing $p_{+}$. We consider a minimization problem on $(M_{+}, g(t))$ as follows: 
\begin{align} \label{deformed minimization}
    \mathcal I_{t}:=\inf\{\mathcal A_{\bar{\theta}}(U): U\in\mathcal E\}.
\end{align}
We denote a minimizer of (\ref{deformed minimization}) by $U_{t}$. Then $\Sigma_{t}:=\partial_{rel}U_{t}$ is achieved with a capillary minimal hypersurface which is disjoint with $\Sigma$ by maximum principle and the convexity condition (5) of Proposition \ref{prop : deformedmetric}. Also by the construction, $\Sigma_{t}$ must intersect $B_{3r_{0}}(p)$ by the metric comparison, otherwise it contradicts to the fact that $\Sigma$ is a minimizer of $\mathcal{A}_{\bar{\theta}}$ in $(M,g)$. If $\Sigma_t$ does not intersect $B_{3r_0}(p)$, then we have 
\begin{align} \label{eqn: functional comparison}
\mathcal A^g_{\bar{\theta}}(\Omega)\leq \mathcal A_{\bar{\theta}}^g(\Omega_t)=\mathcal A_{\bar{\theta}}^{g(t)}(\Omega_t)< \mathcal A_{\bar{\theta}}^{g(t)}(\Omega)\leq \mathcal A_{\bar{\theta}}^g(\Omega),
\end{align}
where we used the fact that $\Sigma$ is strictly mean convex at one interior point with respect to the metric $g(t)$. Suppose $\Sigma_{t}$ does not intersect $B_{r_0}(p)$. Then $R_{g(t)} \ge 0$ along $\Sigma_{t}$ and there exists a point $q \in \Sigma_{t} \cap (B_{3r_0}(p)\setminus B_{r_0}(p))$ with $R_{g(t)}(q) > 0$ by the condition (4) in Proposition \ref{prop : deformedmetric}. This contradicts the rigidity in Theorem \ref{thm: rigidity of Gamma} and Proposition \ref{prop : rigidity sigma}, which gives $R_{g(t)} \equiv 0$ on $\Sigma_{t}$. Hence, $\Sigma_{t}$ intersects $B_{r_0}(p)$.

For small $\epsilon>0$, there exists a family of capillary hypersurfaces $\{ \Sigma_{t} \}_{t \in (0,\epsilon)}$. By Theorem \ref{thm : curv estimate}, $\Sigma_{t}$ smoothly converges to area-minimizing capillary minimal hypersurface $\Sigma'$ in $(M_{+},g)$ which is disjoint to $\Sigma$ since $\Sigma'$ intersects $B_{r_{0}}(p)$. Note that $\Sigma'$ is an infinitesimally rigid hypersurfaces in $(M,g)$ by the rigidity theorems.

With different choices of small $r_{0}$ and $p$, we can repeat the previous procedure to obtain $\Sigma'$ and let $\Sigma^{\rho}$ be a corresponding infinitesimally rigid hypersurface with $\rho = r_{0}$ in $(M,g)$. Let us take $\mathcal{C}$ by a dense subset of $(0,\epsilon)$ and consider $\{ \Sigma^{\rho} \}_{\rho \in \mathcal{C}}$. Then $\cup_{\rho \in \mathcal{C}} \Sigma^{\rho}$ is dense near $\Sigma$ and each surface $\Sigma^{\rho}$ is an infinitesimally rigid area minimizing capillary hypersurface. Hence, we can assume that there exists $\Sigma^{\rho}$ such that $\Sigma^{\rho}$ intersects $B_{\epsilon}(q)$ for any $q \in \partial \Sigma$ and $\epsilon>0$. By the curvature estimate (Theorem \ref{thm : curv estimate}), there exists uniform constant $C$ such that $\partial \Sigma^{\rho}$ intersects $B_{C \epsilon}(q)$ which only depends on $M$ and $g$. This gives a boundary denseness of $\{ \Sigma^{\rho} \}_{\rho \in \mathcal{C}}$ on a small neighborhood of $\Sigma$ on $M_{+}$.

Let us take a vector field $X$ on $\Sigma$ such that $\langle X, N \rangle =1$ on $\Sigma$ and tangential to $\partial M$ on $\partial \Sigma$ and extend to $M$ and denote this vector field by $Y$. We take a one-parameter family of diffeomorphism $\phi_{t}$ on $M$ which is generated by $Y$. For $\rho \in \mathcal{C}$, each $\Sigma^{\rho}$ can be represented as a graph by
\begin{equation*}
    \{ \phi_{u^{\rho}(x)} (x) \, : \, x \in \Sigma\}.
\end{equation*}
We fix a point $x_{0} \in \Sigma$ and define $d_{\rho} = u^{\rho}(x)/ u^{\rho}(x_{0})$. By arguing as in \cite{liu20133}, which computes the first variation of the second fundamental form of $\Sigma_{\rho}$, and by taking the limit $f$ of $d_{\rho}$ as $\rho \rightarrow 0$, we obtain
\begin{equation} \label{weingarten}
    \nabla_{\Sigma}^{2}f(Z,W) +  \text{Rm}_{M}(N,Z,W,N)f= 0 
\end{equation}
for all tangential vectors $Z, W$ on $\Sigma$. By tracing (\ref{weingarten}), we have $\Delta_{\Sigma} f=0$ since $\text{Ric}_{M}(N,N)=0$ by infinitesimal rigidity. Moreover, since $f$ is induced by Jacobi field on $\Sigma$, we have a boundary condition as follows by applying $|A_{\Sigma}| \equiv 0$ from infinitesimal rigidity:
\begin{equation} \label{Robinboundary}
    \partial_{\eta} f = \left( \frac{1}{\sin \overline{\theta}_1} A_{\partial M} (\nu,\nu) + \frac{1}{\sin^{2} \overline{\theta}_1} \partial_{\nu} \cos \overline{\theta}_1 \right)f.
\end{equation}
Moreover, $\partial_{\eta} f = 0$ follows from the rigidity conditions of the stability inequalities in Proposition \ref{prop: first slicing variation}, \ref{prop: second slicing variation formula}, and Lemma \ref{prop: Angle Identities}. By Green's formula, we obtain that $f$ is a constant and $\text{Rm}_{M}(N,Z,W,N)=0$. By Codazzi equation, we obtain $\text{Rm}_{M} \equiv 0$ on $\Sigma$. Since $\{ \Sigma^{\rho} \}_{\rho \in \mathcal{C}}$ is a dense subset in a small neighborhood of $\Sigma$ together with a boundary curvature condition of Theorem \ref{thm: rigidity of Gamma} and Proposition \ref{prop : rigidity sigma}, there exists a small neighborhood of $\Sigma$ that is isometric to Euclidean domain (up to rescaling). We can iterate this procedure along all $M_{+}$ and obtain global isometry by applying these arguments to $M \setminus M_{+}$.
\end{proof}

\subsection{Global rigidity : Hyperbolic case} We prove the rigidity theorem in the hyperbolic case. Since the second slicing is a (weighted) minimal slicing, the global rigidity of $\Gamma$ in $\Sigma$ directly follows from the rigidity proof of Euclidean rigidity. We generalize the construction of a dense set of horospherical capillary slices on both sides of $\Sigma$.
\begin{proof}[Proof of Theorem Hyperbolic rigidity] We assume $\Sigma^{n-1} \subseteq M^{n}$ to be a capillary CMC surface with mean curvature $H_{\Sigma} = -c(n-1)$, the contact angle by $\bar{\theta}$, and the extreme point by $p_{+}$. We take again the convention of the region of $M$ separated by $\Sigma$ that contains $p_{+}$ by $M_{+}$. Also, we denote the unit normal vector field on $\Sigma$ by $N$. We assume $A_{\Sigma} = -cI$ and $\text{Ric}_{M}(N,N)=-(n-1)c^{2}$ by the infinitesimal rigidity condition. 

We fix $p\in M_+$ and the conformal deformation in the hyperbolic case in Proposition and denote the following metric by $g(t)$. Let $\mathcal E$ be the class of admissible finite-perimeter sets. We consider the capillary functional $\mathcal{A}_{\bar{\theta}}$ as
\[
\mathcal A_{\bar{\theta}}(\Omega):=|\partial^i\Omega|-c(n-1)\text{Vol}_g(\Omega)-\int_{\partial^b \Omega}\cos\bar{\theta} d\mathcal H^{n-1}.
\]
We consider a minimization problem (\ref{deformed minimization}) and take a minimizer $U_t$ and $\Sigma_{t}=\partial_{rel}U_{t}$. Note that $H_{\Sigma_t}^{g(t)}=-(n-1)c$. Then $\Sigma_{t}$ is disjoint with $\Sigma$ by the maximum principle and the convexity condition (5) of Proposition \ref{prop : deformedmetric hyp}.

We now prove that $\Sigma_{t}$ must intersect $B_{3r_{0}}(p)$ by the contradiction argument, since the metric comparison does not give (\ref{eqn: functional comparison}) directly as in the Euclidean case. Since we have $g(t) = u^{2}g$ for some smooth function $u : M \rightarrow \mathbb{R}_{+}$ with $u \le 1$, we have $-c(n-1)\text{Vol}_g(\Omega_{t}) \le -c(n-1)\text{Vol}_{g(t)}(\Omega_{t})$ and hence $\mathcal A_{\bar{\theta}}^g(\Omega_t) \leq\mathcal A_{\bar{\theta}}^{g(t)}(\Omega_t)$. Let us consider the first order Taylor expansion of $u$ in terms of $t$ as $u = 1+tv$ where $v \le 0$. Then we have
\begin{align*}
 \mathcal A_{\bar{\theta}}^{g(t)}(\Omega)-\mathcal A_{\bar{\theta}}^g(\Omega) &= \int_\Sigma\left(u^{\frac{2(n-1)}{n-2}}-1\right)\,d \mathcal{H}^{n-1} -c(n-1)\int_{\Omega}\left(u^{\frac{2n}{n-2}}-1\right)\,d\mathcal{H}^{n} \\
 &= \int_\Sigma\frac{2(n-1)}{n-2}tv\,d \mathcal{H}^{n-1} -c(n-1)\int_{\Omega}\frac{2n}{n-2}tv\,d\mathcal{H}^{n} + O(t^{2}) \\
 &\le -C_1tr_0^{n+1} + C_2ctr_0^{n+2} +O(t^{2}) \le 0.
\end{align*}
The last inequality holds since integrands are supported in $B_{3r_{0}}(p)$ for small $r_{0}$ as we chose in Proposition \ref{prop : deformedmetric hyp}. By taking small $t$, we complete the following inequality.
 \[A^g_{\bar{\theta}}(\Omega)\leq \mathcal A_{\bar{\theta}}^g(\Omega_t)\le\mathcal A_{\bar{\theta}}^{g(t)}(\Omega_t)< \mathcal A_{\bar{\theta}}^{g(t)}(\Omega)\leq \mathcal A_{\bar{\theta}}^g(\Omega)
 \]
 and obtain a contradiction, and hence we see that $\Sigma_{t}$ must intersect $B_{3r_{0}}(p)$. Moreover, we obtain that $\Sigma_{t}$ must intersect $B_{r_{0}}(p)$ by the condition (4) in Proposition \ref{prop : deformedmetric hyp} as in the proof of Theorem \ref{main rigidity thm}. By the same arguments in Theorem \ref{main rigidity thm}, we obtain a dense set of infinitesimally rigid capillary CMC hypersurfaces near $\Sigma$$\{ \Sigma^{\rho} \}_{\rho \in \mathcal{C}}$ and boundary denseness of its family by the curvature estimate in Theorem \ref{thm : curv estimate}.

 By taking $X$, $Y$, $\phi_{t}$ as in the proof of Theorem \ref{main rigidity thm}, we represent $\Sigma^{\rho}$ by $\{ \phi_{u^{\rho}(x)} (x) \, : \, x \in \Sigma\}$ a graph and define $d_{\rho} = u^{\rho}(x)/ u^{\rho}(x_{0})$. After passing to a subsequence, $d_\rho$ converges smoothly to a positive function $f$ on $\Sigma$. Since each $\Sigma^\rho$ is infinitesimally rigid and satisfies $A_{\Sigma^\rho}=-cI$, the first variation of the second fundamental form gives
\[
\nabla_\Sigma^2 f(Z,W) + \left( \text{Rm}_M(N,Z,W,N)+c^2g(Z,W) \right)f =0
\]
for all tangential vectors $Z,W$. By taking the trace and plugging it into the infinitesimal rigidity relation $\text{Ric}_M(N,N)=-3c^2$, we obtain
\[
\Delta_\Sigma f + f\bigl(\text{Ric}_M(N,N)+3c^2\bigr) =\Delta_\Sigma f=0.
\]
Since $f$ is induced by Jacobi field along $\Sigma$, we have a boundary condition as follows:
\begin{equation} \label{Robinboundary}
    \partial_{\eta} f = \left( \frac{1}{\sin \overline{\theta}_1} A_{\partial M} (\nu,\nu) - c \cot \bar{\theta}_{1}+ \frac{1}{\sin^{2} \overline{\theta}_1} \partial_{\nu} \cos \overline{\theta}_1 \right)f,
\end{equation}
where we applied $A_{\Sigma} =-cI$ from infinitesimal rigidity. By applying the rigidity conditions of the stability inequalities agai, we obtain $\partial_{\eta} f = 0$ on $\partial \Sigma$. By Green's formula, $f$ is constant and we have $\text{Rm}_M(N,Z,W,N)=-c^2g(Z,W)$ for all $Z$ and $W$ on $\Sigma$. Hence, $M$ has constant sectional curvature $-c^{2}$ along $\Sigma$.

Since $\{ \Sigma^{\rho} \}_{\rho \in \mathcal{C}}$ is a dense subset in a small neighborhood of $\Sigma$ together with a boundary curvature condition of Theorem \ref{thm: rigidity of Gamma} and Proposition \ref{prop : rigidity sigma hyperbolic}, there exists a small neighborhood of $\Sigma$ that has constant sectional curvature $-c^{2}$ in a neighborhood of $\Sigma$. We can iterate this procedure along the whole $M_{+}$ and obtain global isometry by applying these arguments to $M \setminus M_{+}$.
\end{proof}
\begin{proof}[Proof of Theorem \ref{main rigidity in general dimension} and Theorem] By assuming the regularity of capillary area functional in higher dimension, the curvature estimate in Section 5.1 follows. The infinitesimal rigidity follows from our estimates in Section 2 and 3. The global rigidity with the construction of a dense set of Euclidean and hyperbolic capillary slices do not depend on the dimension as far as the curvature estimate holds (See \cite{li2024dihedral} also), Theorem \ref{main rigidity in general dimension} and holds with the assumptions in the statement.
\end{proof}
\appendix \section{Conformally deformed metric near infinitesimally rigid hypersurfaces} \label{appendixconf}
The following proposition follows from an extension of Carlotto-Chodosh-Eichmair \cite{carlotto2016effective}, see also Li \cite{li2024dihedral} for adaption in the high dimensional free boundary setting. We include the construction for the sake of completeness here.
\begin{prop}[Liu \cite{liu20133}, Carlotto-Chodosh-Eichmair \cite{carlotto2016effective}, Li \cite{li2024dihedral}] \label{prop : deformedmetric}
    Suppose $\Sigma^{n-1}$ is an infinitesimally rigid hypersurface in $M^{n}$. Suppose $r_0>0$ is a fixed sufficiently small positive parameter and $p\in \Omega$ is a fixed interior point such that $\text{dist}_g(p,\Sigma)\in (1.5r_0,2.5r_0)$,$\text{dist}_g(p,\partial M)>4r_0$ and $\frac{\partial r}{\partial N_1}>0$ on $\Sigma\cap B_{3r_0}(p)$, where $r(x) = \text{dist}_{g}(x,p)$ and $N_{1}$is the unit normal to $\Sigma$ pointing outward from $\Omega$.  Then there exists an $\epsilon>0$ and a family of Riemannian metrics $\{g(t)\}_{t\in(0,\epsilon)}$ in the conformal class of $g$ with the following properties:
    \begin{enumerate}
        \item $g(t)\to g$ smoothly as $t$ goes to $0$;
        \item $g(t)=g$ on $M\setminus B_{3r_0}(p)$;
        \item $g(t)\leq g$ as metrics on $M$, and the inequality is strict on the annuli $B_{3r_0}(p)\setminus B_{r_0}(p)$;
        \item $R_{g(t)}>0$ on $B_{3r_0}(p)\setminus B_{r_0}(p)$;
        \item $\Sigma$ is weakly mean convex and strictly mean convex at one interior point with respects to the metric $g(t)$ and the unit normal vector field pointing outward from $\Omega$.
    \end{enumerate}
\end{prop}
\begin{proof}
    Let $r_0$ be small enough such that on any geodesic ball $B_{3r_0}(q)$, the inequality $\Delta_g(\text{dist}_g^2(\cdot,q))<4n$ holds. The existence of $r_0$ follows from the compactness of $M$. Let us take a smooth function $f: \mathbb{R} \rightarrow \mathbb{R}$ whose support is in $[0,3]$ and satisfies
    \begin{equation*}
    f(s) = -\text{exp}\left( \frac{16n+8}{s}\right)
    \end{equation*}
    for $s \in (1,3)$. Note that $f$ satisfies $f'(s) >0$ and $(4n-1) f'(s) + sf''(s) <0$ in $(1,3)$. We consider a function $v : M \rightarrow \mathbb{R}$ given by
    \begin{equation*}
        v(x) = r_{0}^{2} f \left( \frac{\text{dist}_{g}(x,p)}{r_{0}}\right).
    \end{equation*}
    In $B_{3r_{0}}(p) \setminus B_{r_{0}}(p)$, we have
    \begin{equation*}
        \Delta_{g}v = r_{0} f'\left( \frac{r}{r_{0}} \right) \Delta_{g}r + f''\left( \frac{r}{r_{0}} \right) < (4n-1) \frac{r_{0}}{r} f'\left( \frac{r}{r_{0}} \right)  + f''\left( \frac{r}{r_{0}} \right)<0.
    \end{equation*}

    Now we take $g(t) = (1+tv)^{\frac{4}{n-2}} g$. The conditions (1)-(3) are satisfied directly and we now check the conditions non-negativeness of scalar curvature (4) and the preservation of weakly convex boundary (5). We have
    \begin{equation*}
        R(g(t)) = (1+tv)^{\frac{n+2}{n-2}} \left( \frac{-4(n-1)t}{n-2} \Delta_{g} v + R(g)(1+tv) \right)>0
    \end{equation*}
    for positive $t$ and we verify (4) in $B_{3r_0}(p)\setminus B_{r_0}(p)$.
    Similarly, by the mean curvature formula by conformal change, along $\Sigma$ we have
    \begin{align*}
        H_{\Sigma,g(t)}=(1+tv)^{-\frac{n}{n-2}}\left((1+tv)H_{\Sigma,g}+\frac{2(n-1)t}{n-2}\frac{\partial v}{\partial N_1}\right)>0,
    \end{align*}
    the condition (5) is verified. 
\end{proof}
Now we prove the hyperbolic version of Proposition \ref{prop : deformedmetric}. 
\begin{prop}[Conformal deformation in hyperbolic setting] \label{prop : deformedmetric hyp}
    Suppose $\Sigma^{n-1}$ is an infinitesimally rigid hypersurface in $M^{n}$ satisfying
    \[
    A_{\Sigma} = -cI \qquad H_{\Sigma} = - (n-1)c.
    \]
    Suppose $r_0>0$ is a fixed sufficiently small positive parameter and $p\in \Omega$ is a fixed interior point such that $\text{dist}_g(p,\Sigma)\in (1.5r_0,2.5r_0)$, $\text{dist}_g(p,\partial M)>4r_0$ and $\frac{\partial r}{\partial N_1}>0$ on $\Sigma\cap B_{3r_0}(p)$, where $r(x) = \text{dist}_{g}(x,p)$ and $N_{1}$is the unit normal to $\Sigma$ pointing outward from $\Omega$.  Then there exist an $\epsilon>0$ and a family of Riemannian metrics $\{g(t)\}_{t\in(0,\epsilon)}$ in the conformal class of $g$ with the following properties:
    \begin{enumerate}
        \item $g(t)\to g$ smoothly as $t$ goes to $0$;
        \item $g(t)=g$ on $M\setminus B_{3r_0}(p)$;
        \item $g(t)\leq g$ as metrics on $M$, and the inequality is strict on the annuli $B_{3r_0}(p)\setminus B_{r_0}(p)$;
        \item $R_{g(t)}+n(n-1)c^2>0$ on $B_{3r_0}(p)\setminus B_{r_0}(p)$;
        \item $\Sigma$ has mean curvature larger or than equal to $-c(n-1)$ and strictly larger than $-c(n-1)$ at one interior point with respect to the metric $g(t)$ and the unit normal vector field pointing outward from $\Omega$.
    \end{enumerate}
\end{prop}
\begin{proof}
    We take the same $f$, $v$, $r$, $u$ as in the proof of Proposition \ref{prop : deformedmetric}. Then we have $\Delta_{g} v <0$ on $B_{3r_0}(p)\setminus B_{r_0}(p)$. We have to check properties (4) and (5). We have
    \begin{align*}
        R(g(t)) + n(n-1)c^{2} &= (1+tv)^{-\frac{n+2}{n-2}}\bigg(-\frac{4(n-1)t}{n-2}\Delta_gv +(1+tv)\bigl(R_g+n(n-1)c^2\bigr) \\ &+n(n-1)c^2\left((1+tv)^{\frac{n+2}{n-2}}-(1+tv)\right)\bigg)
    \end{align*}
    and since
\begin{equation*}
        \Delta_{g}v = r_{0} f'\left( \frac{r}{r_{0}} \right) \Delta_{g}r + f''\left( \frac{r}{r_{0}} \right) < (4n-1) \frac{r_{0}}{r} f'\left( \frac{r}{r_{0}} \right)  + f''\left( \frac{r}{r_{0}} \right)<0,
    \end{equation*}
$-\frac{4(n-1)t}{n-2}\Delta_gv$ dominates $(1+tv)^{\frac{n+2}{n-2}}-(1+tv)$ and we obtain $R(g(t)) + n(n-1)c^{2}>0$.
Similarly, 
\begin{align*}
        H_{\Sigma,g(t)} +(n-1)c=(1+tv)^{-\frac{n}{n-2}}\left((n-1)c((1+tv)^{\frac{n}{n-2}}-(1+tv))+\frac{2(n-1)t}{n-2}\frac{\partial v}{\partial N_1}\right),
    \end{align*}
and $\partial_{N_{1}} v= r_{0} f'(r/r_{0}) \partial_{N_{1}} r$ and this dominates and we obtain $H_{\Sigma,g(t)} +(n-1)c$. Thus, we verified the properties (4) and (5). 
\end{proof}

\bibliographystyle{plain}
\bibliography{refenrence}

\end{document}